\documentclass[11pt,reqno]{amsart}

\usepackage[utf8]{inputenc}
\usepackage[english]{babel}

\usepackage{graphicx}
\usepackage{epstopdf}
\usepackage{enumerate}
\usepackage{enumitem}
\usepackage{amsthm}
\usepackage{amsmath}
\allowdisplaybreaks[1] %% That's for 'align' environment to split into two pages when needed
\usepackage{amsfonts}
\usepackage{mathrsfs}
\usepackage{subcaption}
\usepackage{eurosym}
\usepackage{dsfont}
\usepackage[colorlinks=true, citecolor=blue, linkcolor=blue]{hyperref}
\usepackage{chngpage}
\usepackage{float}
\usepackage{listings}
\usepackage{color}
\usepackage{tikz}

\definecolor{darkgreen}{rgb}{0.0, 0.2, 0.13}
\definecolor{darkpastelgreen}{rgb}{0.01, 0.75, 0.24}
\usepackage{array}
\usepackage{booktabs}
\DeclareMathOperator*{\argmax}{arg\,max}
\DeclareMathOperator*{\erf}{erf}
\DeclareMathOperator*{\erfc}{erfc}

\DeclareMathOperator*{\acoth}{acoth}

\DeclareMathOperator*{\atan}{atan}
\DeclareMathOperator*{\cov}{\mathbb{C}ov}
\DeclareMathOperator*{\var}{\mathbb{V}ar}

\newtheorem{theorem}{Theorem}
\newtheorem{proposition}{Proposition}

\newtheorem{lemma}{Lemma}

\newtheorem{remark}{Remark}
\newtheorem{corollary}{Corollary}

\newcommand{\ind}[1]{\mbox{\rm\large  1}_{\{#1\}}}

\newcommand{\R}{\mathbb{R}}
\newcommand{\N}{\mathbb{N}}
\newcommand{\Z}{\mathbb{Z}}
\newcommand{\p}{\mathbb{P}}
\newcommand{\e}{\mathbb{E}}
\newcommand{\mathcalF}{\mathcal{E}}

\newcommand{\ep}{\varepsilon}

\newcommand{\eqd}{\overset{\rm d}{=}}

\newcommand{\convLs}{\overset{{\mathcal L}^2}{\to}}

\newcommand{\n}{{(n)}}

\newcommand{\QED}{\hfill $\Box$}
\newcommand{\Hspan}{[\underline{H}, \overline{H}]}

\renewcommand{\bar}{\overline}

\makeatletter
\def\env@dmatrix{\hskip -\arraycolsep
  \let\@ifnextchar\new@ifnextchar
  \extrarowheight=2ex
  \array{*\c@MaxMatrixCols{>{\displaystyle}c}}}

\makeatother

\allowdisplaybreaks

\renewcommand{\tilde}{\widetilde}

\date{\today}

\makeatletter
\@namedef{subjclassname@2020}{%
  \textup{2020} Mathematics Subject Classification}
\makeatother

\begin{document}

\title[Derivatives of sup-functionals of fractional Brownian motion]{Derivatives of sup-functionals of fractional Brownian motion evaluated at $H=\tfrac{1}{2}$.}

\author[K.\ Bisewski]{Krzysztof Bisewski}
\address{Department of Actuarial Science, University of Lausanne, UNIL-Dorigny, 1015 Lausanne, Switzerland}
\email{Krzysztof.Bisewski@unil.ch}

\author[K. D\c{e}bicki]{Krzysztof D\c{e}bicki}
\address{Mathematical Institute, University of Wroc{\l}aw, pl. Grunwaldzki 2/4, 50-384 Wroc{\l}aw, Poland}
\email{Krzysztof.Debicki@math.uni.wroc.pl}

\author[T.\ Rolski]{Tomasz Rolski}
\address{Mathematical Institute, University of Wroc{\l}aw, pl. Grunwaldzki 2/4, 50-384 Wroc{\l}aw, Poland}
\email{Tomasz.Rolski@math.uni.wroc.pl}


\begin{abstract}
We consider a family of sup-functionals of (drifted) fractional Brownian motion with Hurst parameter
$H\in(0,1)$. This family includes, but is not limited to: expected value of the supremum, expected workload, Wills functional, and Piterbarg-Pickands constant. Explicit formulas for the derivatives of these functionals as functions of Hurst parameter evaluated at $H=\tfrac{1}{2}$ are established. In order to derive these formulas, we develop the concept of derivatives of fractional $\alpha$-stable fields introduced by Stoev \& Taqqu (2004) and propose Paley-Wiener-Zygmund representation of fractional Brownian motion.
\end{abstract}

\maketitle

%---------------------------------------------------------
%---------------------------------------------------------
%---------------------------------------------------------

\section{Introduction}\label{s:introduction}
Sup-functionals of the supremum of fractional Brownian motion
play crucial role in many problems arising both in theoretical probability and its applications, as in e.g. statistics, financial mathematics, risk theory, queueing theory, see e.g.
\cite{AdT07,DeM15, Man07, Pit96}.
Unfortunately, despite intensive research on their properties, apart from some particular cases --- reduced mostly to standard Brownian motion --- the exact value of such functionals is not known.

Let $ \{B_H (t): t \in \R_+ \} $ be fractional Brownian motion (fBm) with Hurst parameter $ H \in (0,1) $,
that is a centered Gaussian process with continuous sample paths a.s. and
\begin{equation}\label{eq:fbm_cov}
\cov(B_H(s), B_H(t)) = \tfrac{1}{2}\Big(|s|^{2H} + |t|^{2H} - |s-t|^{2H}\Big)
\end{equation}
for all $s,t\in\R_+$. In this manuscript, we consider the following two families of functionals %of the supremum of the fBm modified by a drift function
%(linear or polinomial), :
\begin{equation}\label{def:MH_PH}
\mathscr M_H(T,a) := \e \Big(\sup_{t\in[0,T]} B_H(t)-a t\Big),
\quad \mathscr P_H(T,a) :=\e\Big(\exp\big\{\sup_{t\in[0,T]} \sqrt{2}B_H(t)-a t^{2H}\big\}\Big),
\end{equation}
where $a\in\R$ is the intensity of the drift and $T\in\R_+\cup\{\infty\}$ is the time horizon.
These functionals cover a range of interesting quantities in the
extreme value theory of Gaussian processes. In particular: %listed below:
\begin{itemize}
\item[(i)] For $a\in\R$ and $T>0$,
the quantity $\mathscr M_H(T,a)$ is the expected value of the workload in the
fluid queue with an fBm input at time $T$ under the assumption that at time $0$ the system starts off empty.
Analogously, if $a>0$ and $T=\infty$, then $\mathscr M_H(\infty,a)$
is the expected stationary workload of a queue with an fBm input, see e.g.
\cite{Man07, DeM15};
\item[(ii)] for $T>0$, the quantity $\mathscr H_H(T) := \mathscr P_H(T,1)$ is known as \emph{Wills functional} (\textit{truncated Pickands constant}),
see e.g. \cite{Vit96, DeH20} and references therein; %(or truncated Pickands constant) \red{reference?};
\item[(iii)] for $a>1$, the quantity $\mathscr P_H(\infty,a)$ is known as \emph{Piterbarg constant}; see e.g.
\cite{PiP78, BDHL18} and references therein;
\item[(iv)] the quantity $\mathscr H_H := \lim_{T\to\infty} \frac{1}{T} \mathscr H_H(T)$ is known as \emph{Pickands constant};
see e.g. \cite{Pit96,PiP78}.
\end{itemize}

The values of functionals $\mathscr M_H(T,a)$ and $\mathscr P_H(T,a)$ are notoriously difficult to find in cases other than $H\in\{\tfrac{1}{2},1\}$.
%when $B_H$ is a standard Brownian motion, and a straight line with a random slope, respectively.
Thus, most of the work is
focused on
finding upper and lower bounds for these quantities (see, e.g. \cite{Sha96,DMR03, DeK08,BDHL18, Bor17, Bor18, BDM21})
or determining their asymptotic behavior in various settings
(as $H$ goes to $0$, $H$ goes to $1$, $T$ grows large, $a$ goes to $0$ etc.);
see, e.g., \cite{harper2017, BDM21}.
We note that in many cases simulation methods do not help in estimation of $\mathscr P_H(T,a)$.
 For example, the second moment of $\mathscr P_H(\infty,a)$ does not exist when $a<2$,
 which makes it very difficult to assess error bounds in this case.
 When $H\to0$, then the approximation error resulted from simulation becomes overwhelming, see \cite{DiY14, DiM15}.

In recent years, Delorme, Wiese and others studied the behavior of the supremum of fractional Brownian motion and its time for $H\approx \tfrac{1}{2}$ using the perturbation theory \cite{DelormeEtAl2017Pickands, DelormeandWiese2016Extreme, DelormeWiese2015Maximum, delorme2016perturbative}. Our work was initially inspired by their result in \cite{DelormeEtAl2017Pickands}, where the following expansion for Pickands constant (see item (iv) above) was derived
\begin{equation}\label{eq:pickands_expansion}
\mathscr H_H = 1-2 (H-\tfrac{1}{2}) \gamma_{\rm E} + O ((H-\tfrac{1}{2}) ^ 2), \quad H\to\tfrac{1}{2}
\end{equation}
where $\gamma_{\rm E}$ is the Euler-Mascheroni constant.

The main goal of this contribution is to develop tools for researching expansions similar to \eqref{eq:pickands_expansion} for functionals introduced in Eq.~\eqref{def:MH_PH}. In particular, we find explicit formulas for the derivatives
of these functionals evaluated at $H=\tfrac{1}{2}$, i.e.
\begin{equation}\label{def:MH_PH_derivatives}
\mathscr M_{1/2}'(T,a) := \frac{\partial}{\partial H} \mathscr M_H(T,a)\Big|_{H=1/2} \quad \text{and} \quad \mathscr P_{1/2}'(T,a) := \frac{\partial}{\partial H}\mathscr  P_H(T,a)\Big|_{H=1/2}.
\end{equation}
For these purposes, we consider a probability space on which fractional Brownian motions with all values of
$H\in(0,1)$ are coupled in a non-trivial way. In particular, we consider the Mandelbrot \& van Ness' (MvN) fractional
Brownian motion introduced in the seminal work \cite {mandelbrot1968fractional};
see Eq.~\eqref{eq:def_X} below for precise definition. Following
\cite{peltier1995multifractional} the realizations of MvN field can be viewed as two-dimensional \emph{random surfaces} $(H,t)\mapsto B_H(t)$, as opposed to one-dimensional trajectories $t\mapsto B_H(t)$ for each fixed $H\in(0,1)$. While these random surfaces are non-differentiable in the time direction (with respect to $t$), they turn out to be smooth functions of the Hurst parameter (with respect to $H$). Therefore, following Stoev and Taqqu \cite{StoevTaqqu04, StoevTaqqu05} we define the derivatives of the MvN field with respect to the Hurst parameter $H$. Intuitively speaking, the concept of $H$-derivative of fBm allows us to rigorously write $\frac{\partial}{\partial H}\e B_H(\tau) = \e  \frac{\partial}{\partial H}B_H(\tau)$, where $\tau$ is some (well-behaved) random time and $\{\frac{\partial}{\partial H}B_H(t), t\in\R_+\}$ is a certain explicitly defined Gaussian process. It will turn out that in our context, the latter expression (i.e. $\e  \frac{\partial}{\partial H}B_H(\tau)$) is tractable (explicitly computable).

We note that Stoev and Taqqu considered a broader class of fractional $\alpha$-stable fields, while for the purposes of this paper we limit ourselves to the Gaussian case, i.e. $\alpha=2$, which corresponds to fractional Brownian motion and, in particular, the MvN field. See also \cite[Eq.~(1.9)]{StoevTaqqu05} with $(a_+,a_-) = (1,0)$ and \cite[Chapter~7.4]{samorodnitsky1994stable} for more information on fractional $\alpha$-stable fields. Focusing on the Gaussian case, we strengthen some of the results derived by the authors. In particular, we show sample path continuity of the derivative fields (Proposition~\ref{prop:PWZ_continuous}) and strengthen \cite[Lemma 4.1]{StoevTaqqu04} to almost sure convergence for all $(H,t)\in(0,1)\times\R_+$ (Proposition~\ref{prop:nth_derivative_limit_PWZ}). These propositions are then used in the proofs of main results.

Finally, we propose a Paley-Wiener-Zygmund (PWZ) representation of MvN field and its derivatives. We show that PWZ field is a continuous modification of MvN field (Proposition~\ref{prop:continuous_modification}) and the difference quotients of $n$-th derivative converge everywhere to $(n+1)$-st derivative almost surely (Proposition~\ref{prop:nth_derivative_limit_PWZ}). The PWZ representation is defined in terms of Riemann integrals, as opposed to stochastic integrals in the original Mandelbrot \& van Ness' definition. In the context of this manuscript this representation is more tractable and allows us to express the main quantities of interest, cf.~\eqref{def:MH_PH_derivatives}, as integrals of elementary functions, which we then calculate explicitly in special cases, see Section~\ref{s:functionals_supremum}.

The manuscript is organized as follows.
In Section~\ref{s:preliminaries} we define fractional Brownian motion through its MvN and PWZ representations. We also recall the facts related to the joint distribution of the supremum of drifted Brownian motion and its argmax. In particular, we introduce the conditional distribution of Brownian motion, conditioned on the value of the supremum and its argmax, which follows the law of the generalized 3-dimensional Bessel bridge.
In Section~\ref{sec:main}  we state main results of this paper. In Theorem 1 we give a formula for $\mathscr M'_{1/2}(T,a)$ and in Theorem 2 for $\mathscr P'_{1/2}(T,a)$ in terms of integrals of explicit elementary functions. More explicit formulae for
$\mathscr M'_{1/2}(T,0)$, $\mathscr M'_{1/2}(\infty,a)$ and $\mathscr P'_{1/2}(\infty,a)$ are given in  Corollary~\ref{cor:explicit_expressions}. Additionally, in Corollary~\ref{pit.in} we show that Piterbarg constants $\mathscr P_H(\infty,a)$ are monotone as functions of Hurst parameter. While this result is not directly related to our topic, it is a direct consequence of Proposition~\ref{pit.ine} and it might be of independent interest.
In Section~\ref{s:MvN} we define and examine the $H$-derivatives of fractional Brownian motion both with the use of MvN and PWZ representations.
In Section~\ref{s:proofs_main_theorems} we give the proofs of main theorems. Since they require quite a lot of different results, throughout this section we introduce various preliminary results.  Some of them, for the argmax of $B_H(t)-at$ and $B_H(t)-at^{2H}$ are presented in Proposition~\ref{prop:tau_uniform_tightness} and Proposition~\ref{prop:tau_continuity} respectively. More technical results are deferred to appendices.
In Appendix~\ref{appendixA} we show the equivalence between $\mathcal L^2$ and Paley-Wiener-Zygmund stochastic integrals. In Appendix~\ref{appendix:proofs} we present proofs of results from Section~\ref{s:preliminaries} and Section~\ref{s:MvN}. In Appendix~\ref{appendix:auxiliary_results} we state results needed in the proof of Theorem~\ref{thm:supremum_derivative_limit} and Theorem~\ref{thm:supremum_derivative_pickands_limit}, which can be also of independent interest such as monotonicity of Piterbarg constants and bounds on moments of the supremum of a Gaussian process at a random time. Finally, in Appendix~\ref{appendix:calculations} we write down all the calculations needed to prove Corollary~\ref{cor:explicit_expressions}.

\iffalse
We will now explain how the derivatives of functionals in \eqref{def:MH_PH} are calculated. We present a simplified draft of the proof of how to find the derivative $\mathscr M'_{1/2}(1,0)$. In what follows, let $m_H$ and $\tau_H$ be the maximum and the time of the maximum of a standard Brownian motion over $[0,1]$, i.e.
$$m_H := \max_{t\in[0,1]} B_H(t), \quad \tau_H := \argmax_{t\in[0,1]} B_H(t).$$
Then, $m_H = B_H(\tau_H)$. It turns out that
\begin{align*}
\mathscr M'_H(1,0) := \partial_H \e(B_H(\tau_H)) = \e \Big(\partial_H B_H(\tau_H)\Big)
\end{align*}
Now, we will show that in fact we have
\begin{align*}
\mathscr M'_H(1,0) = \e \Big(\frac{\partial}{\partial H} B_H(\tau_{1/2})\Big)
\end{align*}
\fi

%---------------------------------------------------------
%---------------------------------------------------------
%---------------------------------------------------------

%\section{Preliminaries}

\section{Preliminaries}\label{s:preliminaries}
Let $\{B(t) : t\in\R\}$ be a standard, two-sided Wiener process and let $\R_+ := [0,\infty)$.
In their seminal paper, \cite{mandelbrot1968fractional}, introduced a family of processes $\{X_H(t), t\in\R_+\}$ for $H\in(0,1)$, where
\begin{equation}\label{eq:def_X}
X_H(t) = \int_{-\infty}^0 \big[(t-s)^{H-\tfrac{1}{2}} - (-s)^{H-\tfrac{1}{2}}\big]{\rm d}B(s) + \int_0^t (t-s)^{H-\tfrac{1}{2}}{\rm d}B(s).
\end{equation}
For each $H\in(0,1)$, $\{X_H(t) : t\in\R_+\}$ is a centered Gaussian process with
\begin{align*}
\cov(X_H(s), X_H(t)) = V(H) \cdot \cov(B_H(s), B_H(t)),
\end{align*}
where, $B_H(t)$ is a fractional Brownian motion introduced in \eqref{eq:fbm_cov}, and
\begin{align*}
V(H) := \int_0^\infty \left((1+s)^{H-\tfrac{1}{2}}-s^{H-\tfrac{1}{2}}\right)^2{\rm d}s + \int_0^1 s^{2H-1}{\rm d}s = \frac{\Gamma(\tfrac{1}{2}+H)\Gamma(2-2H)}{2H\Gamma(\tfrac{3}{2}-H)},
\end{align*}
see e.g. \cite[Appendix A]{mishura2008stochastic},
where the explicit formula for $V(H)$ is derived.
This shows that, up to the scaling factor, for each $H\in(0,1)$, process $\{X_H(t):t\in\R_+\}$ is a fractional Brownian motion. Therefore, we call $\{X_H(t):(H,t)\in(0,1)\times\R_+\}$ the \textit{Mandelbrot \& van Ness} (MvN) fractional Brownian field. At the same time, there exists another representation $\{\tilde X_H(t) : (H,t)\in(0,1)\times\R_+\}$ of MvN field,
\begin{equation}\label{def:PWZ_n0}
\begin{split}
\tilde X_H(t) & := (H-\tfrac{1}{2})\int_{-\infty}^0 \Big[(t-s)^{H-\tfrac{3}{2}}-(-s)^{H-\tfrac{3}{2}}\Big]B(s){\rm d}s + t^{H-\tfrac{1}{2}}B(t)\\
&\quad - (H-\tfrac{1}{2})\int_0^t (t-s)^{H-\tfrac{3}{2}}\big(B(t)-B(s)\big){\rm d}s,
\end{split}
\end{equation}
which we call the \textit{Paley-Wiener-Zygmund} (PWZ) representation. In Section~\ref{s:MvN} it is shown that field $\{\tilde X_H(t) : (H,t)\in(0,1)\times\R_+\}$ is a modification of MvN field, whose sample paths $(H,t)\mapsto \tilde X_H(t)$ are continuous almost surely; see Proposition~\ref{prop:continuous_modification}.

From now on, the fractional Brownian motion is defined through the process
$X_H(\cdot)$, i.e. with $D(H) := (V(H))^{-1/2}$, we put
\begin{equation}\label{eq:fbm}
B_H(t) := D(H) \cdot X_H(t).
\end{equation}
For any $a\in\R$ we define
\begin{equation*}
Y_H(t; a) := B_H(t) - a t, \quad \text{and} \quad Y^*_H(t; a) := B_H(t) - at^{2H}.
\end{equation*}
%which are fractional Brownian motions with linear and power drifts respectively.
Notice that $Y_{1/2}^* = Y_{1/2}$ for all $a\in\R$. Furthermore, we define suprema and their locations of these processes
\begin{align*}
\overline Y_H(T, a) & := \sup_{t\in[0,T]} Y_H(t;a), \quad \tau_H(T,a) := \argmax\{t\in[0,T] : Y_H(t; a)\}, \\
\overline Y^*_H(T, a) & := \sup_{t\in[0,T]} Y^*_H(t;a), \quad \tau^*_H(T,a) := \argmax\{t\in[0,T] : Y^*_H(t; a)\}.
\end{align*}
It is known that when $a\in\R$ and $T<\infty$, or $a>0$ and $T=\infty$ then $\overline Y$, $\overline Y^*$ are almost surely finite and $\tau_H$ and $\tau^*_H$ are well-defined (unique) and almost surely finite; see \cite{Ferger1999uniqueness} citing \cite{Lifshits1982absolute}.

Next, we recall the formulae for the joint density of the supremum of (drifted) Brownian motion over $[0,T]$ and its argmax. In the following, for $z\in\R$ we define error function and complementary error function respectively:
\begin{align*}
\erf(z) := \frac{2}{\sqrt{\pi}} \int_0^z e^{-t^2}{\rm d}t, \quad \erfc(z) := 1-\erf(z).
\end{align*}

For brevity of exposition, in this section we write $Y(t) := Y_{1/2}(t;a)$, $\overline Y(T) = \overline Y_{1/2}(T,a)$, and $\tau(T) := \tau_{1/2}(T,a)$. When $a\in\R$ and $T>0$, according to  \cite{shepp1979joint} we have:
\begin{align*}
\frac{\p(\tau(T)\in{\rm d}t, \bar Y(T)\in{\rm d}y, Y(T)\in{\rm d}x)}{{\rm d}t\,{\rm d}y\,{\rm d}x} = \frac{1}{\pi} \frac{y(y-x)}{t^{3/2}(T-t)^{3/2}}\exp\left\{-\frac{y^2}{2t} - \frac{(y-x)^2}{2(T-t)}\right\}e^{-ax - a^2T/2}.
\end{align*}
After the integration of the above density with respect to $x$ over the domain $x<y$ we recover the joint density of the pair $(\tau(T), \bar Y(T))$
\begin{align}\nonumber
p(t,y; T,a) & := \frac{\p(\tau(T)\in{\rm d}t, \bar Y(T)\in{\rm d}y)}{{\rm d}t\,{\rm d}y}\\
\label{eq:density:T} & = \frac{ye^{-\tfrac{(y+ta)^2}{2t}}}{\sqrt{\pi}t^{3/2}} \cdot\left(\frac{e^{-a^2(T-t)/2}}{\sqrt{\pi(T-t)}} + \frac{a}{\sqrt{2}}\cdot \erfc\big(-a\sqrt{\tfrac{T-t}{2}}\big)\right)
\end{align}
for $(t,y)\in(0,T)\times\R_+$; see also \cite[2.1.13.4]{borodin2002handbook}. When $a>0$, then the pair $(\tau(\infty), \bar Y(\infty))$ is well-defined, with density
\begin{equation}\label{eq:density:infty}
p(t,y;\infty,a) = \frac{\sqrt{2}\,aye^{-\tfrac{(y+ta)^2}{2t}}}{\sqrt{\pi}t^{3/2}}
\end{equation}
for $(t,y)\in\R^2_+$; see e.g. \cite[2.1.13.4(1)]{borodin2002handbook}.

\iffalse
Here, we quote the result due to \cite{fitzsimmons2013excursions}, see also \cite[Prop.~2]{asmussen1995discretization}. In what follows, for $t>0, y>0$ let ${\rm BB}(3,t,y)$ denote a law of 3-dimensional Bessel Bridge process from $(0,0)$ to $(t,y)$, that is
\begin{align*}
{\rm BB}(3,t,y) \overset{d}{=} \{R(s), s\in[0,t] \mid R(t) = y\}
\end{align*}
and $\{R(t)\}_{t\in\R}$ is a 3-dimensional Bessel process.

\begin{proposition}\label{prop:BB_def}
Let $Y(t) = B(t)+at$ with $a\in\R$ and $\tau := \argmax\{t\in[0,T] : Y(t)\}$ with $T>0$. Then, conditionally on $Y(\tau) = y, \tau = t$, the process $\{y-Y(t-s)\}_{s\in[0,t]}$ is a ${\rm BB}(3,t,y)$ process.
\end{proposition}
\begin{proof}
See \cite[Prop.~2]{asmussen1995discretization}.
\end{proof}
\fi

When $a\in\R$ and $T>0$ or $a>0$ and $T=\infty$ then, conditionally on $\bar Y(T) = y, \tau(T) = t$, the process $\{Y(t)-Y(t-s)\}_{s\in[0,t]}$ has the law of the \textit{3-dimensional Bessel bridge from $(0,0)$ to $(t,y)$}, see e.g. \cite[Prop.~2]{asmussen1995discretization}. The law of this process does not dependent on the value of the drift $a$ nor on the time horizon $T$. Moreover, the density of the marginal distribution of this process is known. In the following, for $t,y>0$ and $s\in(0,t)$ we define
\begin{equation*}
g(x,s;t,y) := \frac{\p(Y(t)-Y(t-s)\in{\rm d}x \mid \tau(T)=t, \bar Y(T)=y)}{{\rm d}x}
\end{equation*}
Then, according to e.g. \cite[Theorem 1(ii)]{imhof1984density} we have
\begin{equation}\label{eq:W_dens}
g(x,s;t,y)  = \frac{\frac{x}{s^{3/2}} \exp\{-\frac{x^2}{2s}\}}{\frac{y}{t^{3/2}}\exp\{-\frac{y^2}{2t}\}} \cdot \frac{1}{\sqrt{2\pi(t-s)}}\left[e^{-\frac{(x-y)^2}{2(t-s)}} - e^{-\frac{(x+y)^2}{2(t-s)}} \right]\ind{x < y}.
\end{equation}
The following functional of the 3-dimensional Bessel bridge from $(0,0)$ to $(t,y)$ will be important later on
\begin{equation}\label{def:I}
I(t,y) := \e \left(\int_0^t \frac{Y(t)-Y(t-s)}{s}{\rm d}s \mid \tau(T) = t, \overline Y(T) = y\right).
\end{equation}
Using the fact that we have explicit formula for the density $g(\cdot)$ in Eq.~\eqref{eq:W_dens}, we can express $I(t,y)$ as a double integral, see the proposition below, whose prove is given in Appendix~\ref{appendix:proofs}.
\begin{proposition}\label{prop:I(t,y)}
For any $t,y>0$ we have
\begin{equation*}
I(t,y) := \sqrt{\frac{2}{\pi}} \cdot ty^{-1} \int_0^\infty\int_0^\infty \frac{x^2}{q(1+q^2)^2}\left(e^{-(x-yq/\sqrt{t})^2/2} - e^{-(x+yq/\sqrt{t})^2/2}\right){\rm d}x{\rm d}q.
\end{equation*}
\end{proposition}

%---------------------------------------------------------
%---------------------------------------------------------
%---------------------------------------------------------

\section{Main results}\label{sec:main}
%Functionals related to the supremum of fBm}
\label{s:functionals_supremum}

%\subsection{Derivatives of sup-related functionals of fBm at $H=1/2$}\label{ss:derivative_supremum}

Recall the definitions of functionals $\mathscr M_H$ and $\mathscr P_H$ introduced in the Introduction in Eq.~\eqref{def:MH_PH}. We can rewrite them in terms of the expectation of random variables $\overline Y$ and $\overline Y^*$, i.e
\begin{equation*}
\mathscr M_H(T,a) = \e \Big(\overline Y_H(T,a)\Big), \quad \mathscr P_H(T,a) = \e \Big(\exp\{\sqrt{2}\cdot\overline Y^*_H(T,\tfrac{a}{\sqrt{2}})\}\Big).
\end{equation*}
We shall derive formulas for the derivative of functions $H\mapsto \mathscr M_H(T,a)$ and $H\mapsto \mathscr P_H(T,a)$ evaluated at $H=\tfrac{1}{2}$. Following Eq.~\eqref{def:MH_PH_derivatives}, in what follows let
\begin{align*}
\mathscr M'_{1/2}(T,a) := \lim_{H\to1/2} \frac{\mathscr M_H(T,a) - \mathscr M_{1/2}(T,a)}{H-\tfrac{1}{2}}, \quad \mathscr P'_{1/2}(T,a) := \lim_{H\to1/2} \frac{\mathscr P_H(T,a) - \mathscr P_{1/2}(T,a)}{H-\tfrac{1}{2}}.
\end{align*}

\begin{theorem}\label{thm:supremum_derivative_limit}
If $a\in\R$ and $T>0$ or $a>0$ and $T=\infty$, then
\begin{equation}\label{eq:supremum_derivative_limit_integral}
\mathscr M'_{1/2}(T,a) = \int_0^T\int_0^\infty\Big(y(1+\log(t)) + at\log(t) - I(t,y){\rm d}s\Big)p(t,y;T,a) {\rm d}y{\rm d}t.
\end{equation}
\end{theorem}

\begin{theorem}\label{thm:supremum_derivative_pickands_limit}
If $a\in\R$ and $T>0$ or $a>1$ and $T=\infty$, then
\begin{equation}\label{eq:supremum_derivative_pickands_limit_integral}
\mathscr P'_{1/2}(T,a) = \int_0^T\int_0^\infty\sqrt{2}\Big(y(1+\log(t)) - \tfrac{a}{\sqrt{2}}\,t\log(t) - I(t,y)\Big)e^{\sqrt{2}y}p(t,y;T,\tfrac{a}{\sqrt{2}}) {\rm d}y{\rm d}t.
\end{equation}
\end{theorem}

In order to gain a better intuitive understanding of the main results, we provide an outline of the proof of Theorem~\ref{thm:supremum_derivative_limit} below;
the proof of Theorem~\ref{thm:supremum_derivative_pickands_limit} will be similar. Full proofs of these theorems are  given in Section~\ref{s:proofs_main_theorems}.

{\it Outline of the proof of Theorem~\ref{thm:supremum_derivative_limit}.} The main part of the proof is to show that
\begin{align*}
\frac{\partial}{\partial H} \Big[\e\, Y_H(\tau_{H}(T,a),a)\Big]\Big\vert_{H=1/2} = \e \left[\frac{\partial}{\partial H}Y_H(\tau_{1/2}(T,a),a)\Big\vert_{H=1/2}\right],
\end{align*}
where the expression on the left is, by definition, equal to $\mathscr M'_{1/2}(T,a)$. To explain in words, we may swap the order of taking the expected value and differentiation in the definition of $\mathscr M'_{1/2}(T,a)$ and swap $\tau_H(T,a)$ with $\tau_{1/2}(T,a)$ above. The derivative $\frac{\partial}{\partial H}Y_H(t,a)$ is understood point-wise, for every fixed $t\in\R_+$. In proofs we will need an $H$-calculus, which is formally introduced and worked out in Section~\ref{s:MvN}. Here we need only the first derivative
$$
X_H^{(1)}(t) = \int_{-\infty}^0 \big[\log(t-s)(t-s)^{H-\tfrac{1}{2}} - \log(-s)(-s)^{H-\tfrac{1}{2}}\big]{\rm d}B(s) + \int_0^t \log(t-s)(t-s)^{H-\tfrac{1}{2}}{\rm d}B(s).$$
In $H$-calculus, the  Leibniz formula is valid and the $H$-derivative at $H=1/2$ of the fBm  $\frac{\partial}{\partial H}B_H(t)$ is $X(t)+X^{(1)}(t)$. This is derived later in Section~\ref{s:MvN}, see \eqref{eq:Bn_linear_combination}.
As soon as the equation above is established, we find that
\begin{equation*}
\frac{\partial}{\partial H}Y_H(\tau_{1/2}(T,a),a)\Big\vert_{H=1/2} = X(\tau) + X^{(1)}(\tau),
\end{equation*}
where $\tau := \tau_{1/2}(T,a)$. Finally, $\mathscr M'_{1/2}(T,a)$ is equal to the expectation of the expression above
and it can be expressed as the definite integral in Eq.~\eqref{eq:supremum_derivative_limit_integral}
using PWZ representations of $X$ and $X^{(1)}$ and the fact that distribution of
Brownian motion conditioned on its supremum and time of supremum is known; see
Section~\ref{s:preliminaries} for more information. \QED\\

We note that the derivatives $\mathscr M_{1/2}'(T,a)$ and $\mathscr P_{1/2}'(T,a)$ in \eqref{eq:supremum_derivative_limit_integral}
and \eqref{eq:supremum_derivative_pickands_limit_integral} are expressed  as definite integrals.
Thus, they can be computed numerically for any drift $a$ and time horizon $T$,
which satisfy the assumptions of Theorem~\ref{thm:supremum_derivative_limit} and
Theorem~\ref{thm:supremum_derivative_pickands_limit} respectively.
In addition, we were able to calculate these derivatives explicitly in special cases; see the corollary below.
\begin{corollary}\label{cor:explicit_expressions} It holds that
\begin{itemize}
\item[(i)] if $T>0$, then $\displaystyle \mathscr M'_{1/2}(T,0) = \sqrt{\frac{2T}{\pi}} \cdot (\log(T)-2)$,
\item[(ii)] if $a>0$, then $\displaystyle \mathscr M'_{1/2}(\infty,a) = -\frac{1}{a}\left(\gamma_{{\rm E}} +\log(2a^2)\right)$,
\item[(iii)] if $a > 1$, then  $\displaystyle \mathscr P'_{1/2}(\infty,a) = -\frac{2a}{a-1} \Big( 1 + (a-2)\log\big(\tfrac{1-a}{a})\Big)$.
\end{itemize}
\end{corollary}

The calculations leading to Corollary~\ref{cor:explicit_expressions} are deferred to Appendix~\ref{appendix:calculations}. From Corollary~\ref{cor:explicit_expressions}(iii), it straightforwardly follows that for any $a > 1$, the function $H \rightarrow \mathscr P_{H}(\infty,a)$ is decreasing in the neighborhood of $H=\frac{1}{2}$
(because $\displaystyle \mathscr P'_{1/2}(\infty,a)<0$).
The following corollary extends this observation to the whole domain $H\in(0,1)$; its proof follows straightforwardly from Proposition~\ref{pit.ine} given in Appendix~\ref{appendix:auxiliary_results}.
\begin{corollary}\label{pit.in}
Suppose that $0<H_1<H_2< 1$.
\begin{itemize}
\item[(i)] If $a\in \R$ and $T>0$, then
\[
\mathscr P_{H_2}(T,a)\le \mathscr P_{H_1}\left(T^{H_2/H_1},a \right).
\]
\item[(ii)] If $a>1$, then
\[
\mathscr P_{H_2}(\infty,a)\le \mathscr P_{H_1}(\infty,a).
\]
\end{itemize}
\end{corollary}

\begin{remark}\label{rem:conjecture}
Using Mathematica software \cite{Mathematica} and applying certain simplifications, we are able to calculate the following limit
\begin{equation}\label{eq:conjecture_ours}
\lim_{T\to\infty} \lim_{H\to1/2} \frac{\mathscr H_H(T)/T - \mathscr H_{1/2}(T)/T}{H-\tfrac{1}{2}} = \lim_{T\to\infty} \frac{\mathscr P_{1/2}'(T,1)}{T} = -2\gamma_{{\rm E}}.
\end{equation}
It is noted that the result above is very much related to the result established in \cite{DelormeEtAl2017Pickands} on the derivative of Pickands constant at $H=\tfrac{1}{2}$, that is
\begin{equation}\label{eq:conjecture_theirs}
\lim_{H\to1/2} \lim_{T\to\infty}  \frac{\mathscr H_H(T)/T - \mathscr H_{1/2}(T)/T}{H-\tfrac{1}{2}} = \lim_{H\to1/2} \frac{\mathscr H_H - \mathscr H_{1/2}}{H-\tfrac{1}{2}} = -2\gamma_{{\rm E}},
 \end{equation}
see also Eq.~\eqref{eq:pickands_expansion}. The difference between Eq.~\eqref{eq:conjecture_ours} and Eq.~\eqref{eq:conjecture_theirs} is the order of the limit operations.
\end{remark}

%---------------------------------------------------------
%---------------------------------------------------------
%---------------------------------------------------------

\section{Mandelbrot \& van Ness' fractional Brownian field}\label{s:MvN}

In this section we provide some properties of Mandelbrot \& van Ness' field, which play important role in the proofs of results given in Section~\ref{sec:main}.

In contrast to the definition of $B_H(t)$, which we have given at the beginning of Section~\ref{s:introduction},
the definition in \eqref{eq:fbm} provides an additional \textit{coupling}
between fBms with different values of $H$. In fact, we will view the process $\{X_H(t), (H,t)\in(0,1)\times\R_+\}$
as a centered Gaussian field and refer to it as \textit{Mandelbrot \& van Ness' field} (MvN).
Realizations of the MvN field are trajectories (surfaces) $(H,t)\mapsto X_H(t)$;
using the standard rules of ${\mathcal L}^2$ theory of stochastic integrals, we find that
\begin{equation}\label{eq:covariance_mvn}
\begin{split}
\cov(X_{H}(t),X_{H'}(t')) & = \int_{-\infty}^0 \big[(t-s)^{H-\tfrac{1}{2}} - (-s)^{H-\tfrac{1}{2}}\big]\big[(t'-s)^{H'-\tfrac{1}{2}} - (-s)^{H'-\tfrac{1}{2}}\big]{\rm d}s\\
&\quad + \int_0^{t\wedge t'} (t-s)^{H-\tfrac{1}{2}}(t'-s)^{H'-\tfrac{1}{2}}{\rm d}s
\end{split}
\end{equation}
for any $H,H'\in(0,1)$ and $t,t'\in\R_+$. Explicit value of covariance function above was found in \cite[Theorem~4.1]{StoevTaqqu06}. While it is a well-known fact that fBm is self-similar, the same holds true for the MvN field, i.e. for any $c>0$ we have
\begin{equation}\label{eq:self_similarity}
\{B_H(ct): (H,t)\in(0,1)\times\R_+\} \eqd \{c^H B_H(t): (H,t)\in(0,1)\times\R_+\},
\end{equation}
where `$\eqd$' stands for the equality of finite-dimensional distributions, see e.g. \cite[Theorem~2.1(c)]{StoevTaqqu04}; this can also be seen by a direct calculation using \eqref{eq:covariance_mvn}.
It has been shown that there exists a continuous modification of the MvN field, see \cite[Theorem~4]{peltier1995multifractional}.

%Before concluding this section,
We remark that for the purposes of this paper the domain of MvN field is $(H,t)\in(0,1)\times\R_+$, however, we note that it can be extended to $(H,t)\in(0,1)\times\R$.

Next, for any $n\in\Z_+$, where $\Z_+$ is the set of non-negative integers, we define the \textit{$n$th derivative of} MvN  field (with respect to the Hurst parameter), to be the stochastic process $\{X^\n_H(t) : (H,t)\in(0,1)\times\R_+\}$, where
\begin{equation}\label{eq:MvN_derivative_field}
X^\n_H(t) = \int_{-\infty}^0 f^\n_H(t,s){\rm d}B(s) + \int_0^t g^\n_H(t,s){\rm d}B(s),
\end{equation}
for any $(H,t)\in(0,1)\times\R_+$, with $f_H(t,s) := (t-s)^{H-\tfrac{1}{2}} - (-s)^{H-\tfrac{1}{2}}$, $g_H(t,s) := (t-s)^{H-\tfrac{1}{2}}$, and
\begin{equation}\label{def:f_H^\n}
\begin{split}
f_H^\n(t,s) & := \frac{\partial^n}{\partial H^n} f_H(t,s) = \log^n(t-s)(t-s)^{H-\tfrac{1}{2}} - \log^n(-s)(-s)^{H-\tfrac{1}{2}},\\
g_H^\n(t,s) & := \frac{\partial^n}{\partial H^n} g_H(t,s) = \log^n(t-s)(t-s)^{H-\tfrac{1}{2}}.
\end{split}
\end{equation}
We follow the convention that $\frac{\partial^0}{\partial x^0} h(x) := h(x)$ for any function $h(x)$. In particular, for $n=0$, the definition \eqref{eq:MvN_derivative_field} is equivalent to the original MvN field \eqref{eq:def_X}, i.e. $X_H^{(0)} = X_H$. Since the case $H=\tfrac{1}{2}$ will be used particularly often throughout this manuscript, we write $X^\n(\cdot) := X_{1/2}^\n(\cdot)$ for brevity. Definition \eqref{eq:MvN_derivative_field} can be found in \cite{StoevTaqqu05}.

Let us emphasize that all derivatives of the MvN field live on the same probability space, and in fact they are jointly Gaussian.
The covariance between random variables $X^{(n)}_H(t)$, and $X^{(n')}_{H'}(t')$ can be found in analogous way to \eqref{eq:covariance_mvn}.
\iffalse{%using the standard rules of ${\mathcal L}^2$ stochastic  integrals, i.e.
\begin{equation*}
\cov(X^{(n)}_H(t),X^{(n')}_{H'}(t')) = \int_{-\infty}^0f^\n_H(t,s)f^{(n')}_{H'}(t',s){\rm d}s + \int_0^{t\wedge t'}g^\n_H(t,s)g^{(n')}_{H'}(t',s){\rm d}s.
\end{equation*}
%We believe that explicit formula for covariance above could be derivated by appropriate differentiation of the result in \cite[Theorem~4.1]{StoevTaqqu06}. However, in this contribution we don't need to derive explicit values, so this is beyond the scope of this manuscript.\footnote{napisalem cos takiego}
%Moreover, each derivative of the MvN field is a centred Gaussian process with stationary increments; see Proposition~\ref{prop:stationarity_of_linear_combination} below, where an even stronger statement is shown.
\fi

Similarly to the previous section, the derivatives of MvN field also have their PWZ representation. For $n\in\Z_+$ we define a \textit{Paley-Wiener-Zygmund} (PWZ) representation of the MvN field and its derivatives $\{\tilde X_H^\n(t) : (H,t)\in(0,1)\times\R_+\}$, where
\begin{equation}\label{def:PWZ}
\begin{split}
\tilde X^\n_H(t) & := -\int_{-\infty}^0 \Big[\frac{\partial}{\partial s}f^\n_H(t,s)\Big]B(s){\rm d}s + g_H^\n(t,0)B(t)\\
&\quad + \int_0^t \Big[\frac{\partial}{\partial s} g_H^\n(t,s)\Big]\big(B(t)-B(s)\big){\rm d}s,
\end{split}
\end{equation}
for any $(H,t)\in(0,1)\times\R_+$, where
\begin{equation}\label{def:ders_f_H^\n}
\begin{split}
& \frac{\partial}{\partial s}f^\n_H(t,s) = \frac{\partial}{\partial s}g^\n_H(t,s) - \frac{\partial}{\partial s}g^\n_H(0,s)\\
& \frac{\partial}{\partial s}g^\n_H(t,s) = -n\log^{n-1}(t-s)(t-s)^{H-\tfrac{3}{2}} - (H-\tfrac{1}{2})\log^{n}(t-s)(t-s)^{H-\tfrac{3}{2}}.
\end{split}
\end{equation}
Again, we recognize that when $n=0$, the definition \eqref{def:PWZ} is equivalent to \eqref{def:PWZ_n0},
i.e. $\tilde X^{(0)}_H = \tilde X_H$.

The following proposition justifies calling the PWZ, the equivalent \emph{representation} of MvN field.

\begin{proposition}\label{prop:continuous_modification}
For all $n\in\Z_+$, the field $\{\tilde X^\n_H(t): (H,t)\in(0,1)\times\R_+\}$ is a continuous modification of $\{X^\n_H(t): (H,t)\in(0,1)\times\R_+\}$.
\end{proposition}

The result in Proposition~\ref{prop:continuous_modification} is a direct consequence of the considerations in Appendix~\ref{appendixA} (Lemma~\ref{lem:PWZ_lemma} in particular) and the proposition below.

\begin{proposition}\label{prop:PWZ_continuous}
For all $n\in\Z_+$,
\begin{align*}
\p\left((H,t) \mapsto \tilde X^\n_H(t) \text{ is a continuous mapping for all } (H,t)\in(0,1)\times\R_+ \right) = 1.
\end{align*}
\end{proposition}

We note that PWZ representation of fractional Brownian motion is defined in terms of Riemann integrals,
as opposed to the MvN representation, which is defined through ${\mathcal L}^2$ stochastic integrals.

Weaker versions of Proposition~\ref{prop:PWZ_continuous} and~\ref{prop:nth_derivative_limit_PWZ}
were derived in \cite{StoevTaqqu05}.
  The rest of this section is
devoted to showing various fundamental properties of the PWZ field; due to Proposition~\ref{prop:continuous_modification},
these results carry over to the MvN field and its derivatives.
All the proofs are deferred to Appendix~\ref{appendix:proofs}.

In a special case $n=0$, Proposition~\ref{prop:PWZ_continuous} was proven by \cite {peltier1995multifractional}.
The following result justifies calling processes $X^\n_H(t)$ for $n>0$ the \textit{derivatives with respect to the Hurst parameter}.

\begin{proposition}\label{prop:nth_derivative_limit_PWZ}
For all $n\in\Z_+$,
\begin{align*}
\p\left(\lim_{\Delta\to0} \frac{\tilde X^\n_{H+\Delta}(t)-\tilde X^\n_{H}(t)}{\Delta} = \tilde X^{(n+1)}_{H}(t) \text{ for all } (H,t)\in(0,1)\times\R_+\right) = 1.
\end{align*}
\end{proposition}

Notice that thanks to Proposition~\ref{prop:nth_derivative_limit_PWZ}, it makes sense to write $\frac{\partial^n}{\partial H^n} X_H(t) = X_H^\n(t)$.
In the following, for any $k\in\N$ let: $a_1,\ldots,a_k\in\R$, $n_1,\ldots,n_k\in\Z_+$, and $H_1,\ldots,H_k\in(0,1)$ and define
\begin{equation}\label{def:xi_linear_combination}
\eta(t) := \sum_{i=1}^k a_i \cdot X^{(n_i)}_{H_i}(t).
\end{equation}

\begin{proposition}\label{prop:stationarity_of_linear_combination}
Gaussian process $\{\eta(t) : t\in\R_+\}$ defined in \eqref{def:xi_linear_combination} is centered and it has stationary increments.
\end{proposition}

For $N\in\Z_+$, $H\in(0,1)$ we define the Taylor sum remainder
\begin{equation}\label{def:R_H(t;N)}
R_H(t;N) := X_H(t) - \sum_{n=0}^{N-1}\frac{X^\n(t)}{n!}\cdot (H-\tfrac{1}{2})^{n}.
\end{equation}
It is noted that $R_H(t;0) = X_H(t)$ and, according to Proposition~\ref{prop:stationarity_of_linear_combination}, $\{R_H(t;N);t\in\R_+\}$ is a centered Gaussian process with stationary increments. Its variance function can also be proven to be H\"{o}lder continuous; see Lemma~\ref{lem:Var(R_H(t,N))} below.
\begin{lemma}\label{lem:Var(R_H(t,N))}
Let $0<\underline H<\overline H<1$ be such that $\tfrac{1}{2}\in\Hspan$. Then, for any $\ep>0$ and $N\in\Z_+$, there exists a constant $C$ such that
\begin{align*}
\e |R_H(t;N) - R_H(s;N)|^2 \leq C(H-\tfrac{1}{2})^{2N} \cdot \left(|t-s|^{2\underline H-\ep} + |t-s|^{2\overline H+\ep}\right)
\end{align*}
for all $t>0$ and $H\in\Hspan$.
\end{lemma}

\begin{remark}[Extension to normalized MvN field]
Let $B_H(t) = B_H^{(0)}(t) = D(H)\tilde X_H(t)$ and $B^\n_H(t) := \frac{\partial}{\partial H^n}B_H(t)$. Noticing that $D(H)$ is a smooth function, for all $n\in\Z_+$ we have
\begin{align*}
\p\left(\lim_{\Delta\to0} \frac{B^\n_{H+\Delta}(t)- B^\n_{H}(t)}{\Delta} = B^{(n+1)}_{H}(t) \text{ for all } (H,t)\in(0,1)\times\R_+\right) = 1
\end{align*}
Moreover, by a simple application of Leibniz formula for the $n$th derivative of product of functions we find that
\begin{equation}\label{eq:Bn_linear_combination}
B_H^\n(t) = \sum_{k=0}^n \binom{n}{k} D^{(n-k)}(H) \cdot \tilde{X}_H^{(k)}(t).
\end{equation}
The value of $D^\n(H)$ at $H=\tfrac{1}{2}$ for all $n\in\Z_+$ can be found by direct calculation; for example
\begin{equation}\label{eq:values_of_D^\n(1/2)}
D^{(0)}(\tfrac{1}{2}) = 1, \quad D^{(1)}(\tfrac{1}{2}) = 1, \quad D^{(2)}(\tfrac{1}{2}) = -1-\frac{\pi^2}{3}, \quad D^{(3)}(\tfrac{1}{2}) = 3-\pi^2-6\,\zeta(3),
\end{equation}
%$\quad D^{(4)}(\tfrac{1}{2}) = 2\pi^2-\frac{\pi^4}{5} - 3(5+8\zeta(3))$
where $\zeta(\cdot)$ is the Riemann-zeta function.
\end{remark}

%---------------------------------------------------------
%---------------------------------------------------------
%---------------------------------------------------------

\section{Proof of Theorem~\ref{thm:supremum_derivative_limit} and Theorem~\ref{thm:supremum_derivative_pickands_limit}}\label{s:proofs_main_theorems}

Before proving Theorem~\ref{thm:supremum_derivative_limit} and Theorem~\ref{thm:supremum_derivative_pickands_limit}
we need to develop some preliminary results.

Recall the definition of $\tau_H(T,a)$ and $\tau_H^*(T,a)$ from the beginning of Section~\ref{s:functionals_supremum}. In the following proposition we establish that the all-time suprema locations $\tau_H(\infty,a)$ and $\tau^*_H(\infty,a)$ on the MvN field are uniformly bounded with probability growing to 1.
\begin{proposition}\label{prop:tau_uniform_tightness}
Let $0 < \underline H < \overline H < 1$ and $a>0$. Then there exist $C, \gamma>0$ such that
\begin{itemize}
\item[(i)] $\p\Big(\sup_{H\in[\underline H,\overline H]}\tau_H(\infty,a) > T\Big) \leq Ce^{-\gamma T^{2-2\overline H}}$,
\item[(ii)] $\p\Big(\sup_{H\in[\underline H,\overline H]}\tau^*_H(\infty,a) > T\Big) \leq Ce^{-\gamma T^{2\underline H}}$
\end{itemize}
for all $T$ sufficiently large.
\end{proposition}
\begin{proof}
Since the proofs of (i) and (ii) follow by the same idea, we focus only on (ii). Observe that, for $T\ge 1$,
\begin{eqnarray}
\p\Big(\sup_{H\in[\underline H,\overline H]}\tau^*_H(\infty,a) > T\Big)
&\le&
\p\Big(\sup_{H\in[\underline H,\overline H], t\ge T} B_H(t)-a t^{2H}>0\Big)\nonumber\\
&=&
\p\Big(\sup_{H\in[\underline H,\overline H], t\ge T} \frac{B_H(t)}{t^{2H}}>a\Big)\nonumber\\
&\le&
\sum_{k=0}^\infty
\p\Big(\sup_{H\in[\underline H,\overline H], t\in [2^kT,2^{k+1}T]} \frac{B_H(t)}{t^{2H}}>a\Big)\nonumber\\
&=&
\sum_{k=0}^\infty
\p\Big(\sup_{H\in[\underline H,\overline H], t\in [1,2]} \frac{B_H(2^kTt)}{(2^kTt)^{2H}}>a\Big)\nonumber\\
&=&
\sum_{k=0}^\infty
\p\Big(\sup_{H\in[\underline H,\overline H], t\in [1,2]} \frac{2^{Hk} T^H}{2^{2Hk}T^{2H}}\frac{B_H(t)}{t^{2H}}>a\Big)\label{ss}\\
&\le&
\sum_{k=0}^\infty
\p\Big(\sup_{H\in[\underline H,\overline H], t\in [1,2]} \frac{B_H(t)}{t^{2H}}>a2^{\underline H k} T^{\underline H}\Big),\nonumber
\end{eqnarray}
where (\ref{ss}) follows from self-similarity of the MvN field, cf.~\eqref{eq:self_similarity}.
By the continuity of the MvN field, in view of Borell inequality (see, e.g. \cite[Theorem~2.1]{Adl90}),
we know that
\[
\e \left(\sup_{H\in[\underline H,\overline H], t\in [1,2]} \frac{B_H(t)}{t^{2H}}\right)<\infty
\]
and for sufficiently large $T$, using that
$\sup_{H\in[\underline H,\overline H], t\in [1,2]}\var\left( \frac{B_H(t)}{t^{2H}}\right)=1$,
\[
\p\Big(\sup_{H\in[\underline H,\overline H], t\in [1,2]} \frac{B_H(t)}{t^{2H}}>a2^{\underline H k} T^{\underline H}\Big)
\le
2\exp\left( -\frac{ a^2    2^{2\underline H k} T^{2\underline H}}{8}   \right).
\]
Thus, there exists $C>0$ such that for sufficiently large $T$
\begin{eqnarray*}
\sum_{k=0}^\infty
\p\Big(\sup_{H\in[\underline H,\overline H], t\in [1,2]} \frac{B_H(t)}{t^{2H}}>a2^{\underline H k} T^{\underline H}\Big)
\le C \exp\left( -\frac{ a^2}{8} T^{2\underline H}  \right).
\end{eqnarray*}
This completes the proof.
\end{proof}

\begin{proposition}\label{prop:tau_continuity}
If $a\in\R$ and $T>0$ or $a>0$ and $T=\infty$, then
\begin{itemize}
\item[(i)] $\lim_{H\to H'} \tau_H(T,a) = \tau_{H'}(T,a)$, a.s.
\item[(ii)] $\lim_{H\to H'} \tau^*_H(T,a) = \tau_{H'}^*(T,a)$, a.s.
\end{itemize}
for any $H' \in(0,1)$.
\end{proposition}
\begin{proof}
In view of Proposition~\ref{prop:PWZ_continuous}, the random bivariate function $(H,t) \mapsto X_H(t)$ is continuous almost surely, this implies that also, for any fixed $a\in\R$, the functions $(H,t) \mapsto Y_H(t;a)$, and $(H,t) \mapsto Y^*_H(t;a)$ are continuous. Now, consider the case $a\in\R, T<\infty$ - in this case, (i) and (ii) follow from the fact that $Y_H(t;a)\to Y_{H'}(t;a)$, $Y^*_H(t;a)\to Y^*_{H'}(t;a)$ converge uniformly on $t\in[0,T]$, so the argmax functionals must also converge, see e.g. \cite[Lemma~2.9]{Seijo2011A}.

We  now show that (i) holds also when $a<0$ and $T=\infty$. Let $\mathcal A\subset(0,1)$ be any compact interval containing $H'$ and for $n\in\N$ let $A_n := \{\sup_{H\in\mathcal A} \tau_H(a,\infty) < n\}$. Moreover, we have $\{\cup_n A_n\} = \{\sup_{H\in\mathcal A} \tau_H(\infty,a) < \infty\}$, which according to Proposition~\ref{prop:tau_uniform_tightness}, is a set of full measure. We thus can write
\begin{align*}
\p(\lim_{H\to1/2} \tau_H(\infty,a) = \tau_{H'}(\infty,a)) & = \p\big( \lim_{H\to H'} \tau_H(\infty,a) = \tau_{H'}(\infty,a); \cup_{n=1}^\infty A_n\big) \\
& = \p\left(\bigcup_{n=1}^\infty \left\{\lim_{H\to H'} \tau_H(\infty,a) = \tau_{H'}(\infty,a); A_n\right\}\right) \\
& = \lim_{n\to\infty}\p\left(\lim_{H\to H'} \tau_H(\infty,a) = \tau_{H'}(\infty,a); A_n\right),
\end{align*}
where in the last line we used the continuity property of  probability measures for increasing sets. Finally, we notice that on the event $A_n$ we have $\tau_H(\infty,a) = \tau_H(n,a)$ for all $H\in\mathcal A$, thus
\begin{align*}
\p\left(\lim_{H\to H'} \tau_H(\infty,a) = \tau_{H'}(\infty,a); A_n\right) & = \p\left(\lim_{H\to H'} \tau_H(n,a) = \tau_{H'}(n,a); A_n\right) = \p(A_n),
\end{align*}
because we have already established that $\tau_H(T,a) \to \tau_{H'}(T,a)$ a.s. for any fixed $T<\infty$. This concludes the proof of (i) because $\p(A_n) \to 1$, as $n\to\infty$. The proof of (ii) is analogous.
\end{proof}

\begin{corollary}\label{coro:J3}
Let $0 < \underline H < \overline H < 1$,  $N\in\Z_+$. If $a\in\R$ and $T>0$ or $a<0$ and $T=\infty$, then for any $n\in\N$,
\begin{itemize}
\item[(i)] $\displaystyle\sup_{H\in\Hspan}\e \left|\frac{R_H(\tau_H(T,a); N)}{(H-\tfrac{1}{2})^{N/2}}\right|^n < \infty$, and
\item[(ii)] $\displaystyle\sup_{H\in\Hspan}\e \left|\frac{R_H(\tau^*_H(T,a); N)|}{(H-\tfrac{1}{2})^{N/2}}\right|^n < \infty$.
\end{itemize}
\begin{proof}[Proof of Corollary~\ref{coro:J3}]
From Lemma~\ref{lem:Var(R_H(t,N))} we know that for every $\ep>0$ there exists $C>0$ such that
\begin{align*}
\sup_{H\in\Hspan} \var R_H(1;N) \leq C (H-\tfrac{1}{2})^2\cdot t^{2(\underline H - \ep)},
\end{align*}
where $\ep>0$ is taken small enough so that $\underline H - \ep > 0$. Now, again, applying Lemma~\ref{lem:Var(R_H(t,N))} to Lemma~\ref{lem:sup_moments} we find that for every $\ep>0$ and $n\in\N$ there exists a constant $C'>0$ such that
\begin{align*}
\sup_{H\in\Hspan} \e \sup_{t\in[0,1]} |R_H(t;N)|^n \leq C (H-\tfrac{1}{2})^n(1 + \mu^{n-1}(\underline H-\ep) + \mu^n(\underline H-\ep)),
\end{align*}
where $\mu(\underline H-\ep) < \infty$. Finally, applying these two bounds and the result in Proposition~\ref{prop:tau_uniform_tightness}(i) to Lemma~\ref{lem:moments_sup_random_time} we find that there exists $C''>0$ such that
\begin{align*}
\sup_{H\in\Hspan} \e|R_H(\tau_H;N)|^n \leq C''(H-\tfrac{1}{2})^{nN/2}.
\end{align*}
The same reasoning applies to $\tau_H^*$, which completes the proof.
\end{proof}
\end{corollary}

%---------------------------------------------------------
%---------------------------------------------------------
%---------------------------------------------------------

\subsection{Proof of Theorem~\ref{thm:supremum_derivative_limit}}
%\begin{proof}
In the following, for breviety let $Y_H(t) := Y_H(t;a)$, $\tau_H := \tau_H(T,a)$, and $Y(t) := Y_{1/2}(t)$, $\tau := \tau_{1/2}$. We have
\begin{align*}
Y(\tau_H) -  Y(\tau) = (Y_H(\tau_H) - Y_H(\tau)) + (Y_H(\tau)-Y(\tau)),
\end{align*}
We split the proof into three parts. In the first part of the proof we show that
\begin{equation}\label{eq:lem1:conv_0}
\lim_{H\to1/2} \e \left(\frac{Y_H(\tau_H) - Y_H(\tau)}{H-\tfrac{1}{2}}\right) = 0.
\end{equation}
In the second part of the proof we show that
\begin{equation}\label{eq:lem1:conv_EB_prim}
\lim_{H\to1/2} \e \left(\frac{Y_H(\tau) - Y(\tau)}{H-\tfrac{1}{2}}\right) = \e\left(\frac{\partial}{\partial H}Y_H(\tau)\Big|_{H=1/2}\right).
\end{equation}
Finally, in the last part of the proof we show that the claim in Theorem~\ref{thm:supremum_derivative_limit} holds.

{\bf Proof that Eq.~\eqref{eq:lem1:conv_0} holds.}
By definition, $Y_H(t) = D(H)X_H(t)+at$, so for any $t>0$, we have
\begin{align*}
Y_H(t) = at(D(H)-1) + D(H)Y(t) + D(H)X^{(1)}(t)(H-\tfrac{1}{2}) + D(H)R_H(t;2),
\end{align*}
where $R_H(t;2) := X_H(t)-\big[X(t)+X^{(1)}(t)(H-\tfrac{1}{2})\big]$; see \eqref{def:R_H(t;N)}. Furthermore,
\begin{align*}
Y_H(\tau_H)-Y_H(\tau) & = a(D(H)-1)(\tau_H-\tau) + D(H)\big[Y(\tau_H)-Y(\tau)\big] \\
& + D(H)(H-\tfrac{1}{2})\big[X^{(1)}(\tau_H)-X^{(1)}(\tau)\big] + D(H)\big[R_H(\tau_H;2)-R_H(\tau;2)\big].
\end{align*}
Since $Y(\tau) - Y(\tau_H) \geq 0$ and $Y_H(\tau_H)-Y_H(\tau)\geq 0$, then
\begin{eqnarray}
\nonumber\lefteqn{0 \leq \frac{Y_H(\tau_H)-Y_H(\tau)}{H-\tfrac{1}{2}} = a(\tau_H-\tau) \cdot \frac{D(H)-1}{H-\tfrac{1}{2}} + D(H)\big[X^{(1)}(\tau_H)-X^{(1)}(\tau)\big]} \\
\label{eq:J1J2J3}&& + D(H)\cdot\frac{R_H(\tau_H;2)-R_H(\tau;2)}{H-\tfrac{1}{2}} =: J_1(H) + J_2(H) + J_3(H).
\end{eqnarray}
So we need to show that $\e J_i(H) \to 0$, as $H\to0$ for $i\in\{1,2,3\}$. According to Proposition~\ref{prop:tau_continuity}(i) we have $J_1(H)\to0$ a.s. In view of de la Vallée-Poussin theorem, in order to show $\e J_1(H)\to0$, it suffices to show that $\sup_{H\in[1/3,2/3]}\e J_1^2 <\infty$. According to Proposition~\ref{prop:tau_uniform_tightness}, for a compact neighborhood of $H=\tfrac{1}{2}$, for example $H\in[\tfrac{1}{3},\tfrac{2}{3}]$, there exist positive constants $C,\gamma,\beta$ such that $\p(\tau_H >T) \leq C e^{-\gamma T^\beta}$, so for all $H\in[1/3,2/3]$:
\begin{align*}
\e\tau_H^2 = \int_0^\infty 2t\,\p(\tau_H>t){\rm d}t \leq 2C\int_0^\infty te^{-\gamma t^\beta} {\rm d}t.
\end{align*}
Since the right-hand-side is finite and independent of $H$, then $\sup_{H\in[1/3,2/3]} \e \tau_H^2 < \infty$ and $\e J_1(H) \to 0$.

Now, according to Proposition~\ref{prop:PWZ_continuous} and Proposition~\ref{prop:tau_continuity}(i) we have $J_2(H)\to0$ a.s. Moreover, combining Proposition~\ref{prop:tau_uniform_tightness}(i) with Lemma~\ref{lem:moments_sup_random_time}, and the fact that all moments of the supremum $\sup_{t\in[0,1]}X^{(1)}(t)$ exist, we find that the second moments, $\e[X^{(1)}(\tau_H)]^2$, $\e[X^{(1)}(\tau)]^2$ are uniformly bounded for all $H$ close enough to $\tfrac{1}{2}$; again, de la Vallée-Poussin theorem it implies that $\e J_2(H)\to0$.

Finally, we observe that
\begin{align*}
\e |J_3(H)| \leq \frac{\e|R_H(\tau_H;2)|}{H-\tfrac{1}{2}}+\frac{\e|R_H(\tau;2)|}{H-\tfrac{1}{2}}.
\end{align*}
Due to Corollary~\ref{coro:J3}, both terms above tend to $0$, as $H\to\tfrac{1}{2}$, therefore $\e J_3(H) \to 0$.

{\bf Proof that Eq.~\eqref{eq:lem1:conv_EB_prim} holds.} By virtue of Proposition~\ref{prop:nth_derivative_limit_PWZ} and Proposition~\ref{prop:tau_continuity}(i) combined, it is clear that
\begin{equation*}
\frac{Y_H(\tau) - Y(\tau)}{H-\tfrac{1}{2}} \to \frac{\partial}{\partial H} Y_H(\tau)\Big|_{H=1/2} \text{ a.s.}
\end{equation*}
Moreover, we recognize that $Y_H(\tau) - Y(\tau) = R_H(\tau;1)$, see the definition in \eqref{def:R_H(t;N)}, so in view of de la Vallée-Poussin theorem, in order to show that \eqref{eq:lem1:conv_EB_prim} holds, it is enough to show that
\begin{align*}
\sup_{H\in[1/3,2/3]}\,\e\left(\frac{R_H(\tau;1)}{H-\tfrac{1}{2}}\right)^2 < \infty,
\end{align*}
which follows from Corollary~\ref{coro:J3}(i).

{\bf Proof that Eq.~\eqref{eq:supremum_derivative_limit_integral} holds.} Notice that for any $t\in\R_+$ it holds that
\begin{equation*}
\frac{\partial}{\partial H}Y_H(t) = \frac{\partial}{\partial H} (B_H(t)-at) = B_H^{(1)}(t).
\end{equation*}
Combining the results in \eqref{eq:lem1:conv_0} and  \eqref{eq:lem1:conv_EB_prim}, the Leibniz formula derived in \eqref{eq:Bn_linear_combination}, and the fact that $D'(1/2)=1$ (cf.~\eqref{eq:values_of_D^\n(1/2)}) we obtain
\begin{equation}\label{eq:supremum_derivative_limit}
\mathscr M'_{1/2}(T,a) = \e \left[X(\tau) + X^{(1)}(\tau)\right].
\end{equation}
Now, using the PWZ representation for $X^{(1)}$ in \eqref{def:PWZ} we find that
\begin{equation*}
X^{(1)}(\tau) = \int_{-\infty}^\tau\left((\tau-s)^{-1}-(-s)^{-1}\right)B(s){\rm d}s + \log(\tau) B(\tau) - \int_0^\tau \frac{B(\tau)-B(s)}{\tau-s}{\rm d}s.
\end{equation*}
Since $\tau$ is independent of $\{B(s):s<0\}$ then the expected value of the first integral is $0$, and
\begin{align*}
\e X^{(1)}(\tau) & = \e \left[\log(\tau) B(\tau) - \int_0^\tau \frac{B(\tau)-B(s)}{\tau-s}{\rm d}s \right] \\
& = \e \left[\log(\tau) (B(\tau)-a\tau) + a\tau\log(\tau) - \int_0^\tau \frac{(B(\tau)-a\tau)-(B(s)-as)}{\tau-s}{\rm d}s - a\tau \right] \\
& = \e \left[\log(\tau) Y(\tau) + a\tau(\log(\tau)-1) - \int_0^\tau \frac{Y(\tau)-Y(s)}{\tau-s}{\rm d}s\right].
\end{align*}
Now, we recognize that
\begin{align*}
\e \left(\int_0^\tau \frac{Y(\tau)-Y(s)}{\tau-s}{\rm d}s\right) & = \e\left(\e\left(\int_0^\tau \frac{Y(t)-Y(s)}{t-s}{\rm d}s\,\Big\vert\, \tau = t, Y(\tau)=y \right)\right)\\
& = \e\left( I(\tau,Y(\tau))\right),
\end{align*}
with $I(t,y)$ defined in \eqref{def:I}. Since $\e X(\tau) = \e(Y(\tau)+a\tau)$, then from \eqref{eq:supremum_derivative_limit} we obtain
\begin{align*}
\mathscr M'_{1/2}(T,a) = \e\Big[Y(\tau)(1+\log(\tau)) + a\tau\log(\tau) - I(\tau,Y(\tau))\Big].
\end{align*}
The claim now follows because $(\tau,Y(\tau))$ has joint density $p(t,y;T,a)$,
cf. Eq.~\eqref{eq:density:T}.
\QED
%\end{proof}

%---------------------------------------------------------
%---------------------------------------------------------
%---------------------------------------------------------

\subsection{Proof of Theorem~\ref{thm:supremum_derivative_pickands_limit}}
In the following, for brevity let $Z_H(t) := \sqrt{2}B_H(t) - at^{2H}$, $\tau^*_H := \tau^*_H(T,\frac{a}{\sqrt{2}})$, so that $\sup_{t\in[0,T]} Z_H(t) = Z_H(\tau^*_H)$. Additionally we denote $Z(t) := Z_{1/2}(t)$, $\tau^* := \tau^*_{1/2}$. Notice that
\begin{align*}
e^{Z_H(\tau^*_H)} -  e^{Z(\tau^*)} = \left(e^{Z_H(\tau^*_H)} - e^{Z_H(\tau^*)}\right) + \left(e^{Z_H(\tau^*)}-e^{Z(\tau^*)}\right).
\end{align*}
We split the proof into three parts. First part of the proof is to show that
\begin{equation}\label{eq:lem1:conv_0_pickands}
\lim_{H\to1/2} \e \left(\frac{e^{Z_H(\tau^*_H)} - e^{Z_H(\tau^*)}}{H-\tfrac{1}{2}}\right) = 0.
\end{equation}
In the second part of the proof we show that
\begin{equation}\label{eq:lem1:conv_pickands_prim}
\lim_{H\to1/2} \e \left(\frac{e^{Z_H(\tau^*)} - e^{Z(\tau^*)}}{H-\tfrac{1}{2}}\right) = \e \left(\frac{\partial}{\partial H} e^{Z_H(\tau^*)}\Big|_{H=1/2}\right).
\end{equation}
Finally, in the last part of the proof we  show that the claim of Theorem~\ref{thm:supremum_derivative_pickands_limit} holds.

{\bf Proof that Eq.~\eqref{eq:lem1:conv_0_pickands} holds.}
Since $Z_H(\tau^*_H) - Z_H(\tau^*) \geq 0$, then, using the mean value theorem we find that
\begin{equation}\label{eq:mvt_exp}
\e\left[e^{Z_H(\tau^*_H)} - e^{Z_H(\tau^*)}\right] \leq \e\left[\Big(Z_H(\tau^*_H)-Z_H(\tau^*)\Big) \cdot e^{Z_H(\tau^*_H)}\right],
\end{equation}
so it suffices to show that the bound above converges to $0$.

By definition $Z_H(t) = \sqrt{2}D(H)X_H(t)-at^{2H}$, so for any $t>0$, we have
\begin{align*}
Z_H(t) = a(D(H)t-t^{2H}) + D(H)Z(t) + \sqrt{2}D(H)\left(X^{(1)}(t)(H-\tfrac{1}{2}) + R_H(t;2)\right),
\end{align*}
where $R_H(t;2) := X_H(t)-\big[X(t)+X^{(1)}(t)(H-\tfrac{1}{2})\big]$ was defined in \eqref{def:R_H(t;N)}. Furthermore,
\begin{align*}
Z_H(\tau^*_H)-Z_H(\tau^*) & = a\left[D(H)(\tau_H^*-\tau^*)-\left((\tau_H^*)^{2H}-(\tau^*)^{2H}\right)\right] + D(H)\big[Z(\tau^*_H)-Z(\tau^*)\big] \\
& + D(H)(H-\tfrac{1}{2})\big[X^{(1)}(\tau^*_H)-X^{(1)}(\tau^*)\big] + D(H)\big[R_H(\tau^*_H;2)-R_H(\tau^*;2)\big].
\end{align*}
Now, for $H\in(0,1)$, $t\in\R_+$ we define
\begin{align*}
U_H(t) := D(H) \cdot \frac{t - t^{2H}}{H-\tfrac{1}{2}} + t^{2H}\cdot\frac{D(H)-1}{H-\tfrac{1}{2}}.
\end{align*}
Since $Z(\tau^*)-Z(\tau_H^*)\geq 0$ and $Z_H(\tau_H)-Z_H(\tau)\geq 0$, then
\begin{align*}
0 \leq \frac{Z_H(\tau_H^*)-Z_H(\tau^*)}{H-\tfrac{1}{2}} & \leq aJ_1(H) + D(H)\Big(J_2(H)+J_3(H)\Big),
\end{align*}
where
\begin{align*}
J_1(H) := U_H(\tau_H^*)-U_H(\tau^*), \quad J_2(H) := X^{(1)}(\tau^*_H)-X^{(1)}(\tau^*), \quad J_3(H) := \frac{R_H(\tau^*_H;2)-R_H(\tau^*;2)}{H-\tfrac{1}{2}}.
\end{align*}
The proofs that $\e J_i(H)\to0$, $i\in\{2,3\}$, as $H\to\tfrac{1}{2}$ are analogous to the corresponding statements in the proof of Theorem~\ref{thm:supremum_derivative_limit}, cf.~Eq.~\eqref{eq:J1J2J3}. Thus, we show only that $\e J_1(H)\to0$, as $H\to\tfrac{1}{2}$. We have
\begin{align*}
\lim_{H\to1/2} U_H(t) = -2t\log(t) + t =: U(t)
\end{align*}
and the convergence is uniform for any compact subset of $\R_+$. Moreover,
\begin{align*}
|U_H(\tau_H^*) - U_H(\tau^*)| \leq |U_H(\tau_H^*) - U(\tau_H^*)| + |U(\tau_H^*)-U(\tau_H)| + |U_H(\tau^*)-U(\tau^*)|,
\end{align*}
where each of the terms of the sum above converges to $0$ due to uniform convergence $U_H(t)\to U_H(t)$, continuity of $U$ and the fact that $\tau_H^*\to\tau^*$ a.s. (see Proposition~\ref{prop:tau_continuity}(ii)), so we have $J_1(H)\to 0$ a.s. Using the mean value theorem we find that, with $\underline H := \min\{\tfrac{1}{2},H\}$, $\overline H := \max\{\tfrac{1}{2},H\}$,
\begin{equation}\label{eq:bound_with_mvt}
\left|\frac{t - t^{2H}}{H-\tfrac{1}{2}}\right| \leq \sup_{\tilde H\in\Hspan}|2\log(t)t^{2\tilde H}| \leq 2|\log(t)|(1+t^{2\overline H}).
\end{equation}
Now, we have that
\begin{equation*}
U_H^2(t) \leq 2D^2(H)\left(\frac{t - t^{2H}}{H-\tfrac{1}{2}}\right)^2 + 2t^{4H}\left(\frac{D(H)-1)}{H-\tfrac{1}{2}}\right)^2,
\end{equation*}
so using the bound in \eqref{eq:bound_with_mvt} and Proposition~\ref{prop:tau_uniform_tightness}(ii), we find that the second moments $\e [U_H(\tau_H^*)]^2, \e [U_H(\tau^*)]^2 < \infty$ are uniformly bounded for all $H$ close enough to $\tfrac{1}{2}$. In view of de la Vallée-Poussin theorem, this observation combined with the fact that $J_1(H)\to0$ a.s. imply that $\e J_1(H)\to0$.

{\bf Proof that Eq.~\eqref{eq:lem1:conv_pickands_prim} holds.} By virtue of Proposition~\ref{prop:nth_derivative_limit_PWZ} and \eqref{prop:tau_continuity} combined, it is clear that
\begin{align*}
\frac{e^{Z_H(\tau^*)} - e^{Z(\tau^*)}}{H-\tfrac{1}{2}} \to \frac{\partial}{\partial H} e^{Z_H(\tau^*)}\Big|_{H=1/2} \quad \text{a.s.}
\end{align*}
In view of de la Vallée-Poussin theorem, it suffices to show that some $1+\ep$ with $\ep>0$, the absolute moment of the pre-limit above is bounded for all $H$ close enough to $1/2$. Using the mean value theorem as in \eqref{eq:mvt_exp} and applying H\"{o}lder's inequality afterwards, we find that
\begin{align*}
\e\left(\frac{e^{Z_H(\tau^*)} - e^{Z(\tau^*)}}{H-\tfrac{1}{2}}\right)^{1+\ep} \leq \sup_{\tilde H\in\Hspan} \e\bigg(|Z_H(\tau^*)-Z(\tau^*)|\cdot \exp\{Z_{\tilde H}(\tau_{\tilde H}^*)\}\bigg)^{1+\ep} \\
\leq \left(\sup_{\tilde H\in\Hspan} \e|Z_H(\tau^*)-Z(\tau^*)|^{(1+\ep)q}\right)^{1/q} \cdot \left(\sup_{\tilde H\in\Hspan} \e\,\exp\{(1+\ep)pZ_{\tilde H}(\tau_{\tilde H}^*)\}\right)^{1/p},
\end{align*}
where $p,q>1$, and $1/p + 1/q = 1$. Now, according to Proposition~\ref{pit.ine}, if we take  $\ep>0, p>1$ small enough to satisfy $(1+\ep)p < a$, then the second term remains bounded, as $H\to1/2$. For the boundedness of the first term, see that $Z_H(\tau^*) - Z(\tau^*) = \sqrt{2}R_H(\tau^*;1)$ and
\begin{align*}
\sup_{\tilde H\in\Hspan}\e|R_H(\tau^*;1)|^{(1+\ep)q} < \infty
\end{align*}
for arbitrarily large $q>1$ due to Corollary~\ref{coro:J3}.

{\bf Proof that Eq.~\eqref{eq:supremum_derivative_pickands_limit_integral} holds.} Notice that for any $t\in\R_+$ it holds that
\begin{align*}
\frac{\partial}{\partial H} e^{Z_H(t)} & = \sqrt{2}\left[\frac{\partial}{\partial H} (B_H(t)-at^{2H})\right]e^{Z_H(t)} \\
& = \sqrt{2}\left(B^{(1)}_H(t) - \frac{2a}{\sqrt{2}}\cdot t^{2H}\log(t)\right)\exp\left(\sqrt{2}B_H(t)-at^{2H}\right).
\end{align*}
Combining the results in \eqref{eq:lem1:conv_0_pickands} and  \eqref{eq:lem1:conv_pickands_prim}, the Leibniz formula derived in \eqref{eq:Bn_linear_combination}, and the fact that $D'(\tfrac{1}{2})=1$, cf.~\eqref{eq:values_of_D^\n(1/2)}, we obtain
\begin{equation*}
\mathscr P'_{1/2}(T,a) = \e \left[\sqrt{2}\left(X(\tau) + X^{(1)}(\tau) - \frac{2a}{\sqrt{2}} \cdot \tau\log(\tau)\right)\exp\left(\sqrt{2}X(\tau)-a\tau\right)\right].
\end{equation*}
The rest of the proof is analogous to the third part of the proof of Theorem~\ref{thm:supremum_derivative_limit}.
\QED
%\end{proof}

%---------------------------------------------------------
%---------------------------------------------------------
%---------------------------------------------------------

\appendix
\section{Paley-Wiener-Zygmund representation of ${\mathcal L}^2$ stochastic integrals}\label{appendixA}
One of the core tools used in this contribution is the equivalence between MvN and PWZ representation of the fractional Brownian field and its derivatives introduced in Section~\ref{s:MvN}. In this appendix we explain how to find the Paley-Wiener-Zygmund representation of ${\mathcal L}^2$ stochastic integral for a general class of processes.

Define formally a stochastic process $\{\xi(t):t\in\R_+\}$
%defined through the ${\mathcal L}^2$ stochastic integral
\begin{equation}\label{def:xi}
\xi(t):= P_1(t)+P_2(t), \quad P_1(t):=\int_{-\infty}^0 \psi(t,s)\,{\rm d}B(s), \ P_2(t):=\int_0^t\phi(t,s)\,{\rm d}B(s).
\end{equation}
For a moment we may think about the integrals in  ${\mathcal L}^2$ sense.
The processes $P_1$ and $P_2$ are well defined if
\begin{equation}\label{eq:warunki}
\int_{-\infty}^0\psi(t,s)^2\,{\rm d}s<\infty,\quad \int_{-\infty}^t\phi(t,s)^2\,ds<\infty, \quad \text{for all} \ t\in\R_+.
\end{equation}

Suppose   $ \{B (t): t \in \R \} $
 is a Brownian motion  on a probability space $(\Omega,\mathcalF,\p)$.
 Let $T<0$. For $k=1,2,\ldots$ let  $M_k:=|\lceil T2^k\rceil|$. Define integral sums for $ P_1(t,T):=\int_{T}^0 \psi(t,s)\,{\rm d}B(s)$ by
\begin{equation}\label{eq:Pk1}
P_{k,1}(t,T):=\sum_{i=1}^{M_k-1}\psi\big(t,-\tfrac{i}{2^k}\big)\Big(B(-\tfrac{i}{2^k}\big)-B\big(-\tfrac{i+1}{2^k}\big)\Big).
\end{equation}
Similarly, set $N_k=\lfloor t2^k\rfloor$ and define  integral sums for $P_2(t)$ by
\begin{equation}\label{eq:Pk2}
P_{k,2}(t)=\sum_{i=0}^{N_k-1}\phi\big(t,\tfrac{i+1}{2^k}\big)\Big(B\big(\tfrac{i+1}{2^k}\big)-B\big(\tfrac{i}{2^k}\big)\Big).
\end{equation}
Clearly we have
\begin{equation}\label{eq:convergence:Pk1Pk2}
P_{k,1}(t,T)\convLs P_1(t,T)=\int_{T}^0\psi(t,s)\,{\rm d}B(s), \quad \text{and} \quad P_{k,2}(t)\convLs P_2(t)=\int_0^t\phi(t,s)\,{\rm d}B(s).
\end{equation}
We denote $\psi'(t,s)=\frac{\partial}{\partial s}\psi(t,s)$
and $\phi'(t,s)=\frac{\partial}{\partial s}\phi(t,s)$. In what follows, we introduce the following assumptions
\begin{itemize}
  \item[(A)] For each $t>0$, the function $s\mapsto\psi(t,s)$ belongs to {$C^1(-\infty,0)$} and there exist constants $C,\ep>0$ such that
  \begin{itemize}
  \item[$\bullet$] $|\psi(t,s)|\le C|s|^{-1/2+\epsilon}$, $|\psi'(t,s)|\le C|s|^{-3/2+\epsilon}$, \ as $s\uparrow 0$,
  \item[$\bullet$] $|\psi(t,T)|\le C|T|^{-1/2-\ep}$, $|\psi'(t,T)|\le C|T|^{-3/2-\ep}$, \ as $T\downarrow-\infty$.
  \end{itemize}
  \item[(B)] For each $t>0$, the function $s\mapsto\phi(t,s)$ belongs to $C^1[0,t)$ and there exist constants $C,\ep>0$ such that
   \begin{itemize}
  \item[$\bullet$] $|\phi(t,s)|\le C|t-s|^{-1/2+\ep}$, $|\phi'(t,s)|\le C|t-s|^{-3/2+\epsilon}$, as $s\uparrow t$.
  \end{itemize}
\end{itemize}
 Note that  conditions (A) and (B) imply \eqref{eq:warunki}.
Conditions (A) and (B) were formulated having in mind 
  $\psi=f_H^{(n)}$ and $\phi=g_H^{(n)}$, which were defined in \eqref{def:f_H^\n}.
\begin{lemma}\label{lem:PWZ_lemma}
  Let functions $\psi(t,s)$, $\phi(t,s)$ satisfy the assumptions (A) and (B)
  respectively.
   Then, with $\xi(t)$ defined in \eqref{def:xi}, for each $t\geq0$,
   $\xi(t)=\tilde{P}_1(t)+\tilde{P}_2(t)$ a.s., where
   $\tilde{P}_1(t)=-\int_{-\infty}^0\psi'(t,s)B(s)\,ds$ and
   $\tilde{P}_2(t)=\phi(t,0)B(t)+\int_0^t\phi'(t,s)(B(t)-B(s))\,ds$.
\end{lemma}

\begin{proof}
  Set $N=\lfloor 2^kt\rfloor$. Then, the partial sum $P_{k,2}(t)$
 in  \eqref{eq:Pk2} can be rewritten as follows
  \begin{eqnarray*}
    \lefteqn{\sum_{i=0}^{N-1}\phi\big(t,\tfrac{i+1}{2^k}\big)\Big(B\big(\tfrac{i+1}{2^k}\big)-B\big(\tfrac{i}{2^k}\big)\Big)}\\
    &&=\sum_{i=1}^{N-1}\Big(\phi\big(t,\tfrac{i}{2^{k}}\big)-\phi\big(t,\tfrac{i+1}{2^{k}}\big)\Big)B\big(\tfrac{i}{2^k}\big)+\phi\big(t,\tfrac{N}{2^k}\big)B\big(\tfrac{N}{2^k}\big)\\
    &&=\sum_{i=1}^{N-1}\Big(\phi\big(t,\tfrac{i}{2^{k}}\big)-\phi\big(t,\tfrac{i+1}{2^{k}}\big)\Big)\Big(B\big(\tfrac{i}{2^k}\big)-B(t)\Big)+\phi\big(t,\tfrac{N}{2^k}\big)B\big(\tfrac{N}{2^k}\big)\\
    &&\qquad +B(t)\sum_{i=1}^{N-1}\Big(\phi\big(t,\tfrac{i}{2^{k}}\big)-\phi\big(t,\tfrac{i+1}{2^{k}}\big)\Big).
  \end{eqnarray*}
  We now use the mean value theorem
  $\phi(t,\frac{i}{2^k})-\phi(t,\frac{i+1}{2^k})=-\frac{1}{2^k}\phi'(t,\theta_{k,i})$ for $\theta_{k,i}\in (\frac{i}{2^k},\frac{i+1}{2^k})$ and hence
   the above equals to
   $$ \frac{1}{2^k}\sum_{i=1}^{N-1}\phi'(t,\theta_{k,i})\Big(B(t)-B\big(\tfrac{i}{2^k}\big)\Big)
   +B(t)\phi\big(t,\tfrac{1}{2^k}\big)+
   \phi\big(t,\tfrac{N}{2^k}\big)\Big(B\big(\tfrac{N}{2^k}\big)-B(t)\Big)=
\tilde{P}_{k,2}(t).$$
Finally, under condition (B), as $k\to\infty$ we have
   $$\tilde{P}_{k,2}(t) \to \tilde{P}_2(t)=\int_0^t\phi'(t,s)(B(t)-B(s))\,ds+\phi(t,0)B(t) \quad \text{a.s.}$$
On the way we use that $B(t)$ is a realization of the Brownian motion, whose trajectories are almost surely H{\"o}lder continuous, i.e. for every $\ep>0$ and compact set $K\subset\R$, there exists some $C>0$ such that
\begin{equation}\label{eq:appendixA_holder}
|B(t)-B(s)|\le C |t-s|^{1/2-\ep} \text{ for all } t,s\in K \quad \text{a.s.}
\end{equation}
 Now, since the integral sums $\tilde{P}_{k,2}(t)$ converge in probability (also in ${\mathcal L}^2$) to $\tilde{P}_2(t)$, cf.~\eqref{eq:convergence:Pk1Pk2}, therefore ${P}_2(t) = \int_0^t\phi'(t,s)(B(t)-B(s))\,ds+\phi(t,0)B(t)$ a.s., which completes the proof for the positive part.

  Now let $T <0$ and set $M=|\lceil T2^k\rceil|$. The integral sums
  for $P_1(t,T)$ in Eq.~\eqref{eq:Pk1}
  can be transformed to
  \begin{eqnarray*}
  \lefteqn{\sum_{i=1}^{M-1}\psi\big(t,-\tfrac{i}{2^k}\big)
    \Big(B\big(-\tfrac{i}{2^k}\big)-B\big(-\tfrac{i+1}{2^k}\big)\Big)}\\
  &&=
\psi\big(t,-\tfrac{1}{2^k}\big)B\big(-\tfrac{1}{2^k}\big)+
  \sum_{i=1}^{M-2}\Big(\psi\big(t,-\tfrac{i+1}{2^k}\big)-\psi\big(t,-\tfrac{i}{2^k}\big)\Big)B\big(-\tfrac{i}{2^k}\big)
  -\psi\big(t,\tfrac{M-1}{2^k}\big)B\big(-\tfrac{M}{2^k}\big)\\
  &&=
\psi\big(t,-\tfrac{1}{2^k}\big)B\big(-\tfrac{1}{2^k}\big)-
  \frac{1}{2^k}\sum_{i=1}^{M-2}\psi'(t,\theta_{k,i})
  B\big(-\frac{i}{2^k}\big)-\psi\big(t,-\tfrac{M-1}{2^k}\big)B\big(-\tfrac{M}{2^k}\big)=\tilde{P}_{k,1}(t,T),
\end{eqnarray*}
for some $\theta_{k,i}\in (-\tfrac{i+1}{2^k},-\tfrac{i}{2^k})$. Now, under condition (A), and due to the H\"{o}lder continuity of Brownian motion, cf.~\eqref{eq:appendixA_holder} we find that, as $k\to\infty$:
\begin{itemize}
\item[$\bullet$] $\psi(t,-\frac{1}{2^k})B(-\frac{1}{2^k})\to0$ \ a.s.,
\item[$\bullet$] $\psi(t,-\frac{M-1}{2^k})B(-\frac{M}{2^k})\to \psi(t,T)B(T)$ \ a.s.,
\item[$\bullet$] $\sum_{i=1}^{M-2}\frac{1}{2^k}\psi'(t,\theta_{k,i})
  B(-\frac{i}{2^k})\to \int_T^0\psi'(t,s)B(s)\,{\rm d}s$ \ a.s.
\end{itemize}
Therefore, as $k\to\infty$ we have
$$P_{k,1}(t,T) \to \tilde{P}_1(t,T)=-\int_T^0\psi'(t,s)B(s)\,{\rm d}s+\psi(t,T)B(T)\quad \text{a.s.}$$
We emphasize that the condition (A) ensures the convergence of the integral around $s=0$. Finally, the condition (A) together with the property for Brownian motion (which is implied by the law of iterated logarithm) that for every $\ep>0$ there exists $C>0$ such that $|B(T)|\le C|T|^{1/2+\epsilon}$ for $T\downarrow -\infty$ almost surely, we have $\psi(t,T)B(T)\to 0$ a.s. for $T\to-\infty$.
Similarly it ensures the convergence of the integral around $s=-\infty$.
The proof is completed.
\end{proof}

\begin{remark}
It follows from the proof of Lemma~\ref{lem:PWZ_lemma} that the integrals $\int \psi(t,s)\,dB(s)$ and $\int \phi(t,s)\,dB(s)$ can be understood in the pointwise sense (as well as in the mean square sense). This is because $P_{k,1}(t,T)=\tilde{P}_{k,1}(t,T)$, $P_{k,2}(t)=\tilde{P}_{k,2}(t)$ and almost surely we have $\tilde{P}_{k,2}(t) \to \tilde{P}_{2}(t)$, and $\tilde{P}_{k,1}(t,T)\to \tilde{P}_{1}(t,T)$. Hence $P_{2}(t)=\tilde{P}_2(t)$, which shows that $P_{2}$ is defined pointwise. Similar argument is valid for $P_1$. 
   Notice that at the endpoints  $s=-\infty$ and $s=t$ there are singularities of $\psi(t,s)$ and $\phi(t,s)$, and therefore we approach to these integrals in the manner of the theory of improper Riemann integrals.
Now let
\begin{equation*}
    \tilde{\xi}(t)=\tilde{P}_1(t)+\tilde{P}_2(t),
\end{equation*}
where
  $$\tilde{P}_1(t)=-\int_{-\infty}^0\psi'(t,s)B(s)\,ds, \quad \tilde{P}_2(t)=\phi(t,0)B(t)+\int_0^t\phi'(t,s)(B(t)-B(s))\,ds.$$
  From the proof of Lemma~\ref{lem:PWZ_lemma} we see that $\xi(t)=\tilde{\xi}(t)$ a.s. for each $t\geq0$.
  Supposing that realizations of processes $\xi$ and $\tilde{\xi}$ are continuous a.s., from  \cite[page 57]{kallenberg2002} we gain that processes $\xi$ and $\tilde\xi$ are indistinguishable. 
  \end{remark}

%---------------------------------------------------------
%---------------------------------------------------------
%---------------------------------------------------------

\section{Proofs}\label{appendix:proofs}

\begin{proof}[Proof of Proposition~\ref{prop:I(t,y)}]
Notice that we may change the order of expectation and integration because the underlying functions are non-negative, and obtain
\begin{align*}
\e \left(\int_0^t \frac{Y(t)-Y(t-s)}{s}{\rm d}s \mid \tau(T) = t, \overline Y(T) = y\right) = \int_0^t \frac{\e\left(Y(t)-Y(t-s)\mid \tau(T) = t, \overline Y(T) = y\right)}{s}{\rm d}s.
\end{align*}
Furthermore, using the formula for the density of generalized Bessel bridge in \eqref{eq:W_dens} we find that
\begin{align*}
I(t,y) & = \int_0^t\int_0^\infty s^{-1}x\cdot g(x,s;t,y)\,{\rm d}x{\rm d}s \\
& = \frac{t^{3/2}}{y} \int_0^t\int_0^\infty\frac{x^2}{s^{5/2}\sqrt{2\pi(t-s)}}e^{-\frac{x^2}{2s}+\frac{y^2}{2t}}\cdot\left[e^{-\frac{(x-y)^2}{2(t-s)}} - e^{-\frac{(x+y)^2}{2(t-s)}}\right]\,{\rm d}x{\rm d}s \\
& = \frac{t^{3/2}}{\sqrt{2\pi}y} \int_0^t\int_0^\infty\frac{x^2}{s^{5/2}\sqrt{t-s}}\cdot\left[e^{-\tfrac{1}{2}\big(\tfrac{x}{\sigma_0}-y\mu_0\big)^2} - e^{-\tfrac{1}{2}\big(\tfrac{x}{\sigma_0}+y\mu_0\big)^2}\right]\,{\rm d}x{\rm d}s,
\end{align*}
where $\mu_0^2 := \frac{s}{(t-s)t}$, and $\sigma_0^2 := \frac{(t-s)s}{t}$. Applying substitution $x := \sigma_0 x$, and $s := tw$ afterwards yields
\begin{align*}
I(t,y) = \frac{t}{\sqrt{2\pi}y} \int_0^1\int_0^\infty\frac{(1-s)x^2}{s}\cdot\left[e^{-\tfrac{1}{2}(x-yq(s)/\sqrt{t})^2} - e^{-\tfrac{1}{2}(x+yq(s)/\sqrt{t})^2}\right]\,{\rm d}x{\rm d}s,
\end{align*}
where $q(s) := \sqrt{\frac{s}{1-s}}$. Substituting $q(s)$ yields the result.
\end{proof}

We introduce a simple lemma, which we give without proof.
\begin{lemma}\label{lem:log_behavior}
For every $n\in\Z_+$ and $\ep>0$ there exists $C>0$ such that $|\log^n(x)| \leq C(x^{\ep} + x^{-\ep})$ for all $x>0$.
\end{lemma}

\begin{proof}[Proof of Proposition~\ref{prop:PWZ_continuous}]
The law of iterated logarithm for Brownian motion implies that
\begin{equation}\label{eq:iterated_logarithm}
\limsup_{t\to\pm \infty} \frac{|B(t)|}{|t|^{1/2+\ep}} =0 \quad \text{ a.s.}
\end{equation}
for every $\ep>0$. Moreover, it is also known that paths of Brownian motion are H\"{o}lder continuous almost surely on any bounded time interval, that is, for any $\ep>0$ there exist finite random numbers $A^\ep_k, n\in\Z_+$ such that
\begin{equation}\label{eq:holder_continuity}
|B(t)-B(s)| \leq A^\ep_k|t-s|^{1/2-\ep} \text{ for all } t,s\in[-n,n] \ \text{ a.s.}
\end{equation}
%Now, the intersection of countably many events of probability 1 is an event of probability 1.
From now on, let us fix a single trajectory of Brownian motion, which satisfies \eqref{eq:iterated_logarithm} and \eqref{eq:holder_continuity} for all $k\in\Z_+$ and for some $\ep>0$ chosen at the end. Let us fix $n\in\N$, $H'\in(0,1)$ and $t'\in \R$. Given the representation \eqref{def:PWZ}, it suffices to show that
\begin{align}
\label{eq:show_continuity_I}
& \lim_{(t,H)\to(t',H')} \int_{-\infty}^0 \Big[\frac{\partial}{\partial s}f^\n_H(t,s)\Big]B(s){\rm d}s = \int_{-\infty}^0 \Big[\frac{\partial}{\partial s}f^\n_{H'}(t',s)\Big]B(s){\rm d}s,\\
\label{eq:show_continuity_II}
& \lim_{(t,H)\to(t',H')} g_H^\n(t,0)B(t) = g_{H'}^\n(t',0)B(t'),\\
\label{eq:show_continuity_III}
& \lim_{(t,H)\to(t',H')} \int_0^t \Big[\frac{\partial}{\partial s}g_H^\n(t,s)\Big](B(t)-B(s)){\rm d}s = \int_0^{t'} \Big[\frac{\partial}{\partial s}g_{H'}^\n(t',s)\Big](B(t')-B(s)){\rm d}s.
\end{align}
Clearly, for each $s\in\R$ we have
\begin{align*}
\lim_{(t,H)\to(t',H')} \frac{\partial}{\partial s}f^\n_H(t,s) = \frac{\partial}{\partial s}f^\n_{H'}(t',s), \quad \text{and} \quad \lim_{(t,H)\to(t',H')} \frac{\partial}{\partial s}g^\n_H(t,s) = \frac{\partial}{\partial s}g^\n_{H'}(t',s),
\end{align*}
so \eqref{eq:show_continuity_II} is satisfied.

We will now show that an integrable majorant for \eqref{eq:show_continuity_I} exists in two domains: (i) $s\in[-s^*,0)$, and (ii) $s<-s^*$, where $s^* := \max\{1,t\}$. Consider case (i) $s\in[-s^*,0)$; we have
\begin{align*}
\left|\frac{\partial}{\partial s}f^\n_H(t,s)\right| & \leq n|\log^{n-1}(t-s)(t-s)^{H-\tfrac{3}{2}}| + |H-\tfrac{1}{2}|\cdot|\log^{n}(t-s)(t-s)^{H-\tfrac{3}{2}}| \\
& +  n|\log^{n-1}(-s)(-s)^{H-\tfrac{3}{2}}| + |H-\tfrac{1}{2}|\cdot|\log^{n}(-s)(-s)^{H-\tfrac{3}{2}}|.
\end{align*}
According to Lemma~\ref{lem:log_behavior}, for every $\delta>0$ and $(H,t)$ close enough to $(H',t')$, there exists $C_1,C_2>0$ such that
\begin{align*}
\left|\frac{\partial}{\partial s}f^\n_H(t,s)\right| \leq C_1 + C_2\left(|s|^{H'-3/2-\delta}\right)
\end{align*}
for all $-s^*<s<0$. Moreover, using \eqref{eq:holder_continuity} with $n=\lceil s^* \rceil$ we find that
\begin{align*}
\left|\frac{\partial}{\partial s}f^\n_H(t,s)B(s)\right| \leq \left(C_1+C_2|s|^{H'-3/2-\delta}\right)A^\ep_n|s|^{1/2-\ep}.
\end{align*}
The upper bound above is integrable over $s\in[-s^*,0)$ when $H'-\delta-\ep > 0$, which completes the proof of case (i). We now consider the case (ii) $s<-s^*$; we have
\begin{equation*}
\left|\frac{\partial}{\partial s}f^\n_H(t,s)\right| \leq n|F_H(t,s;n-1)| + (H-\tfrac{1}{2})|F_H(t,s;n)|,
\end{equation*}
where $F_H(t,s;k) := \log^k(t-s)(t-s)^{H-\tfrac{3}{2}}-\log^k(-s)(-s)^{H-\tfrac{3}{2}}$. Using the mean value theorem we find that for $s<-s^*\leq -t$ we have
\begin{align*}
|F_H(t,s;k)| & \leq t\sup_{\tilde s\in[-s,t-s]}\bigg|k\log^{k-1}(\tilde s)(\tilde s)^{H-5/2} + (H-\tfrac{3}{2})\log^k(\tilde s)(\tilde s)^{H-5/2}\bigg| \\
& \leq t \cdot \left(k\log^{k-1}(-2s)(-s)^{H-5/2} + |H-\tfrac{3}{2}|\cdot\log^k(-2s)(-s)^{H-5/2}\right).
\end{align*}
According to Lemma~\ref{lem:log_behavior}, for any $\delta>0$ and $H$ close enough to $H'$, the above is upper bounded by $C|s|^{H'+\delta-5/2}$ for some $C>0$ and all $|s|\geq s^*$. Using \eqref{eq:iterated_logarithm} we thus find that there exists some $C'>0$ such that
\begin{align*}
\left|\frac{\partial}{\partial s}f^\n_H(t,s)B(s)\right| \leq C'|s|^{H'-\delta-5/2}|s|^{1/2+\ep}.
\end{align*}
The upper bound above is an integrable function over $s<-s^*$ when $H'+\delta +\ep < 1$, which completes  the proof of case (ii).

It is left to show \eqref{eq:show_continuity_III}. Let $\eta>0$ and let $(H,t)$ be close enough to $(H',t')$ such that $|t'-t| < \eta/2$. We split the pre-limit integral \eqref{eq:show_continuity_II} into two domains (i) $s\in[0,t'-\eta]$, and (ii) $s\in[t'-\eta,t]$. It is clear that
\begin{align*}
\lim_{(t,H)\to(t',H')} \int_0^{t'-\eta} \Big[\frac{\partial}{\partial s}g_H^\n(t,s)\Big](B(t)-B(s)){\rm d}s = \int_0^{t'-\eta} \Big[\frac{\partial}{\partial s}g_{H'}^\n(t',s)\Big](B(t')-B(s)){\rm d}s
\end{align*}
for all $\eta$ small enough by the Lebesgue dominated convergence theorem with dominant being the supremum of $[\frac{\partial}{\partial s}g_H^\n(t,s)](B(t)-B(s))$ over $s\in[0,t'-\eta]$, which is uniformly bounded for all $(t,H)$ sufficiently close to $(t',H')$. It thus suffices to show that
\begin{align*}
\lim_{\eta\to 0}\lim_{(t,H)\to(t',H')} \int_{t'-\eta}^t \Big[\frac{\partial}{\partial s}g_H^\n(t,s)\Big](B(t)-B(s)){\rm d}s = 0.
\end{align*}
Since we have
\begin{align*}
\left|\frac{\partial}{\partial s}g_H^\n(t,s)\right| \leq n|\log^{n-1}(t-s)(t-s)^{H-\tfrac{3}{2}}| + |H-\tfrac{1}{2}|\cdot|\log^{n}(t-s)(t-s)^{H-\tfrac{3}{2}}|,
\end{align*}
then according to Lemma~\ref{lem:log_behavior}, for any $\ep>0$ we can find a constant $C$ such that
\begin{align*}
\left|\frac{\partial}{\partial s}g_H^\n(t,s)\right| \leq C|t-s|^{H-\tfrac{3}{2}-\ep}.
\end{align*}
The above combined with the H\"{o}lder continuity of the Brownian trajectory, with $n=\lceil t'\rceil+1$, for any $\delta>0$ there exists some constant $C'>0$ such that
\begin{align*}
\left|\Big[\frac{\partial}{\partial s}g_H^\n(t,s)\Big](B(t)-B(s))\right| \leq C'|t-s|^{H-\tfrac{3}{2}-\ep+1/2-\delta}.
\end{align*}
This gives us
\begin{align*}
\left|\int_{t'-\eta}^t \Big[\frac{\partial}{\partial s}g_H^\n(t,s)\Big](B(t)-B(s)){\rm d}s\right| \leq C'\int_0^{|t-t'|+\eta}s^{H-1-\ep-\delta}{\rm d}s = \frac{C'}{H-\ep-\delta}(2\eta)^{H-\ep-\delta}
\end{align*}
so after taking $\ep+\delta$ small enough, we find that the above converges to $0$, as $\eta\to0$.
\end{proof}

\begin{proof}[Proof of Proposition~\ref{prop:nth_derivative_limit_PWZ}]
Let us fix $n\in\N$, $H'\in(0,1)$, and $t\in \R_+$. Given the representation \eqref{def:PWZ}, it suffices to show that
\begin{align}
\label{eq:show_derivative_I}
& \lim_{\Delta\to0} \int_{-\infty}^0 \frac{\frac{\partial}{\partial s}f^\n_{H+\Delta}(t,s)-\frac{\partial}{\partial s}f^\n_H(t,s)}{\Delta} \cdot B(s){\rm d}s = \int_{-\infty}^0 \Big[\frac{\partial}{\partial s}f^{(n+1)}_{H'}(t',s)\Big]B(s){\rm d}s,\\
\label{eq:show_derivative_II}
& \lim_{\Delta\to0} \frac{\frac{\partial}{\partial s}g^\n_{H+\Delta}(t,0)-\frac{\partial}{\partial s}g^\n_H(t,0)}{\Delta} = g_{H}^{(n+1)}(t,0),\\
\label{eq:show_derivative_III}
\begin{split}
& \lim_{\Delta\to0} \int_0^t \frac{\frac{\partial}{\partial s}g^\n_{H+\Delta}(t,s)-\frac{\partial}{\partial s}g^\n_H(t,s)}{\Delta} \cdot (B(t)-B(s)){\rm d}s = \int_0^{t} \Big[\textstyle\frac{\partial}{\partial s}g_{H}^{(n+1)}(t,s)\Big](B(t)-B(s)){\rm d}s.
\end{split}
\end{align}
Clearly for each $s\in\R$ we have
\begin{align*}
& \lim_{\Delta\to0} \frac{\frac{\partial}{\partial s}f^\n_{H+\Delta}(t,s)-\frac{\partial}{\partial s}f^\n_H(t,s)}{\Delta} = \frac{\partial}{\partial s}f^{(n+1)}_{H}(t,s),\\
& \lim_{\Delta\to0} \frac{\frac{\partial}{\partial s}g^\n_{H+\Delta}(t,s)-\frac{\partial}{\partial s}g^\n_H(t,s)}{\Delta} = \frac{\partial}{\partial s}g^{(n+1)}_{H}(t,s),
\end{align*}
so \eqref{eq:show_derivative_II} is satisfied. Using the mean value theorem we find that, for every $s<0$ there exists $\tilde H_s\in[H,H']\cup[H',H]$ and for every $s>0$ there exists $\hat H_s\in[H,H']\cup[H',H]$ such that
\begin{equation}\label{eq:fg_mvt_H}
\begin{split}
& \frac{\big|\frac{\partial}{\partial s}f^\n_{H+\Delta}(t,s)-\frac{\partial}{\partial s}f^\n_H(t,s)\big|}{\Delta} = f^{(n+1)}_{\tilde H_s}(t,s), \frac{\big|\frac{\partial}{\partial s}g^\n_{H+\Delta}(t,s)-\frac{\partial}{\partial s}g^\n_H(t,s)\big|}{\Delta} = g^{(n+1)}_{\hat H_s}(t,s).
\end{split}
\end{equation}
Using the representation in \eqref{eq:fg_mvt_H}, the proofs of \eqref{eq:show_derivative_I} and \eqref{eq:show_derivative_III} are now analogous the the proofs of \eqref{eq:show_continuity_I} and  \eqref{eq:show_continuity_III} of Proposition~\ref{prop:PWZ_continuous}.
\end{proof}

\begin{proof}[Proof of Lemma~\ref{lem:Var(R_H(t,N))}]
Since $R_H(t;N)$ has stationary increments; cf. Proposition~\ref{prop:stationarity_of_linear_combination}, then $\e |R_H(t;N) - R_H(s;N)|^2 = \var R_H(|t-s|;N)$, so it suffices to bound the variance $\var R_H(t;N)$ for all $t>0$. We have
\begin{eqnarray*}
\lefteqn{R_H(t;N) = \int_{-\infty}^0 \left( f_H(1,s) - \sum_{n=0}^N \frac{f_{1/2}^\n(t,s)}{n!}(H-\tfrac{1}{2})^n \right) {\rm d}B(s)} \\
&&+ \int_0^t \left( g_H(t,s) - \sum_{n=0}^N \frac{g_{1/2}^\n(t,s)}{n!}(H-\tfrac{1}{2})^n \right) {\rm d}B(s).
\end{eqnarray*}
Thus, $\var \big(R_H(t;N)\big) = I_1 + I_2$, where
\begin{equation}\label{eq:I1I2_show_bound}
\begin{split}
I_1 & := \int_0^\infty \left[ f_H^\n(t,s) - \sum_{n=0}^{N-1} \frac{f_{1/2}^\n(t,s)}{n!}(H-\tfrac{1}{2})^n \right]^2 {\rm d}s\\
I_2 & := \int_0^1 \left[ g_H^\n(t,s) - \sum_{n=0}^{N-1} \frac{g_{1/2}^\n(t,s)}{n!}(H-\tfrac{1}{2})^n \right]^2 {\rm d}s.
\end{split}
\end{equation}
Now, let $\underline H := \min\{\tfrac{1}{2},H\}$, $\overline H := \max\{\tfrac{1}{2},H\}$. Using Taylor's theorem with Lagrange remainder we find that
\begin{equation*}
\begin{split}
I_1 \leq (H-\tfrac{1}{2})^{2N}\frac{1}{(N!)^2}\int_0^\infty \sup_{\tilde H\in\Hspan}\left(f_{\tilde H_s}^{(N)}(t,s)\right)^2{\rm d}s \\
I_2 \leq (H-\tfrac{1}{2})^{2N}\frac{1}{(N!)^2}\int_0^t \sup_{\tilde H\in\Hspan} \left(g_{\tilde H}^{(N)}(t,s)\right)^2{\rm d}s.
\end{split}
\end{equation*}
Consider the integral in the upper bound for $I_1$ above on two different domains: (i) $s\in(0,t]$, and (ii) $s\in(t,\infty)$. In case (i) we have
\begin{eqnarray*}
\lefteqn{
\sup_{\tilde H\in\Hspan}\left(f_{\tilde H_s}^{(N)}(1,s)\right)^2 = \sup_{\tilde H\in\Hspan}\left(\log(t+s)(t+s)^{\tilde H-\tfrac{1}{2}}-\log(s)s^{\tilde H-\tfrac{1}{2}}\right)^2{\rm d}s} \\
&& \leq 2\left(\sup_{\tilde H\in\Hspan}\log^{2N}(t+s)(t+s)^{2\overline H-1} + 2\sup_{\tilde H\in\Hspan} \log^{2N}(s)s^{2\underline H-1}\right).
\end{eqnarray*}
Now, using Lemma~\ref{lem:log_behavior} and the bound above, we find that for every $\ep>0$ there exists $C>0$ such that
\begin{eqnarray*}
\sup_{\tilde H\in\Hspan}\left(f_{\tilde H_s}^{(N)}(1,s)\right)^2 &\leq& C\left(\sup_{\tilde H\in\Hspan}s^{2H-1-\ep} + \sup_{\tilde H\in\Hspan}s^{2H-1+\ep}\right) \\
&\leq & 2C\left(s^{2\underline H-1-\ep} + s^{2\overline H-1+\ep}\right).
\end{eqnarray*}
This gives us that there exists some constant $C_*>0$ such that
\begin{align*}
\int_0^t \sup_{\tilde H\in\Hspan}\left(f_{\tilde H_s}^{(N)}(1,s)\right)^2{\rm d}s \leq C_* \left(t^{2\underline H-\ep} + t^{2\overline H+\ep}\right),
\end{align*}
which completes case (i). For case (ii), using the mean value theorem we find that
\begin{eqnarray*}
\lefteqn{f_{H}^{(N)}(t,s) \leq t\cdot \sup_{\tilde s\in[s,t+s]}\left|\frac{\partial}{\partial s}\log^N(s)s^{H-\tfrac{1}{2}}\right|_{s=\tilde s}} \\
&& \leq t\cdot \sup_{\tilde s\in[s,t+s]}|N\log^{N-1}(\tilde s)(\tilde s)^{H-\tfrac{3}{2}}| + t\cdot \sup_{\tilde s\in[s,t+s]}|(H-\tfrac{1}{2})\log^{N}(\tilde s)(\tilde s)^{H-\tfrac{3}{2}}|.
\end{eqnarray*}
Again, using Lemma~\ref{lem:log_behavior} and the bound above, we find that for every $\ep>0$ there exists $C>0$ such that
\begin{align*}
f_{H}^{(N)}(t,s) \leq Ct\cdot \left(\sup_{\tilde s\in[s,t+s]}(\tilde s)^{H-\tfrac{3}{2}-\tfrac{\ep}{2}} + \sup_{\tilde s\in[s,t+s]}(\tilde s)^{H-\tfrac{3}{2}+\tfrac{\ep}{2}}\right),
\end{align*}
so for sufficiently small $\ep$, the suprema
above are attained at $\tilde s = s$ and thus
\begin{align*}
f_{H}^{(N)}(t,s) \leq Ct(s^{H-\tfrac{3}{2}-\tfrac{\ep}{2}}+s^{H-\tfrac{3}{2}+\tfrac{\ep}{2}}).
\end{align*}
Furthermore, this gives us that there exists some constant $C_*>0$ such that
\begin{align*}
\int_{t}^\infty \sup_{\tilde H\in\Hspan}\left(f_{\tilde H_s}^{(N)}(1,s)\right)^2{\rm d}s & \leq C_* \sup_{\tilde H\in\Hspan} t^2(t^{2\tilde H-2-\ep} + t^{2\tilde H-2+\ep})\\
& \leq  2C_*(t^{2\underline H-\ep} + t^{2\overline H+\ep}),
\end{align*}
which concludes the proof of case (ii) and shows that there exists $C>0$ such that
\begin{align*}
I_1 \leq (H-\tfrac{1}{2})^{2N}\frac{C}{(N!)^2}(t^{2\underline H-\ep} + t^{2\overline H+\ep})
\end{align*}
for all $t>0$, $H\in\Hspan$. The upper bound for $I_2$ from \eqref{eq:I1I2_show_bound} can be found using analogous method, as the upper bound for $I_1$ in case (i) above.
\end{proof}

\begin{proof}[Proof of Proposition~\ref{prop:stationarity_of_linear_combination}]
Let $t,s\in\R_+$, then
\begin{align*}
\e (\eta(t)-\eta(s))^2 = \sum_{i=1}^k\sum_{j=1}^k a_i a_j\e \Big(X^{(n_i)}_{H_i}(t) - X^{(n_i)}_{H_i}(s)\Big)\Big(X^{(n_j)}_{H_j}(t) - X^{(n_j)}_{H_j}(s)\Big).
\end{align*}
Without the loss of generality, assume that $t>s\geq 0$. It suffices to show that each term of the double sum above depends only on $t-s$. Now, notice that for any $H\in(0,1)$, $n\in\Z_+$ and $t>s>0$ we have
\begin{equation*}
\begin{split}
X^\n_H(t)-X^\n_H(s) & = \int_{-\infty}^s \Big[g^\n_H(t,w) - g^\n_H(s,w) \Big]{\rm d}B(w) + \int_s^t g^\n_H(t,w) {\rm d}B(w),
\end{split}
\end{equation*}
thus, for any $H,H'\in(0,1)$, $n,n'\in\Z_+$, and $t>s$ we obtain
\begin{eqnarray*}
\lefteqn{\e \Big(X^{(n)}_{H}(t) - X^{(n)}_{H}(s)\Big)\Big(X^{(n')}_{H'}(t) - X^{(n')}_{H'}(s)\Big)} \\
&&= \int_{-\infty}^s \Big(g^\n_H(t,w)-g^\n_H(s,w)\Big)\Big(g^{(n')}_{H'}(t,w)-g^{(n')}_{H'}(s,w)\Big)\,{\rm d}w + \int_s^t g^\n_H(t,w)g^{(n')}_{H'}(t,w)\,{\rm d}w,
\end{eqnarray*}
After applying the substitution $z := w - s$ to both integrals above we find that
\begin{eqnarray*}
\lefteqn{\e \Big(X^{(n)}_{H}(t) - X^{(n)}_{H}(s)\Big)\Big(X^{(n')}_{H'}(t) - X^{(n')}_{H'}(s)\Big)} \\
&&= \int_{-\infty}^0 f^\n_H(t-s,z)f^{(n')}_{H'}(0,z)\,{\rm d}z + \int_0^{t-s} g^\n_H(t-s,z)g^{(n')}_{H'}(t-s,z)\,{\rm d}z,
\end{eqnarray*}
whose value depends only on $t-s$.
\end{proof}

\section{Auxiliary results on extremes of Gaussian processes}\label{appendix:auxiliary_results}

We present  some useful properties of extremes of Gaussian processes, which are also of independent interest.
\begin{proposition}\label{pit.ine}
Suppose that $0<H_1<H_2\le 1$.
\\ (i) For $T>0$, $a\in \R$ and $\alpha\in \R$
\[
\e \left( \exp\left( \alpha \left(\sup_{t\in [0,T]}  \sqrt{2}B_{H_2}(t)-at^{2H_2} \right)   \right) \right)
\le
\e \left( \exp\left( \alpha \left( \sup_{t\in [0,T^{H_2/H_1}]}  \sqrt{2}B_{H_1}(t)-at^{2H_1}   \right) \right) \right)
< \infty.
\]
(ii) For $a>1$ and $\alpha< a$
\[
\e \left( \exp\left( \alpha \left(\sup_{t\in [0,\infty)}  \sqrt{2}B_{H_2}(t)-at^{2H_2} \right)   \right) \right)
\le
\e \left( \exp\left( \alpha \left( \sup_{t\in [0,\infty)}  \sqrt{2}B_{H_1}(t)-at^{2H_1}   \right) \right) \right)
< \infty.
\]
\end{proposition}
\begin{proof}
We note that
$\var(B_{H_2}(t))=\var\left(B_{H_1}\left(t^{H_2/H_1} \right)\right)$
and
\begin{eqnarray}
\cov(B_{H_2}(s),B_{H_2}(t))
&=&
\frac{1}{2}\left( s^{2H_2}+t^{2H_2}-|s-t|^{2H_2}\right)\nonumber\\
&\ge&
\frac{1}{2}\left( s^{2H_2}+t^{2H_2}-|s^{H_2/H_1}-t^{H_2/H_1}|^{2H_1}\right)\label{i.conv}\\
&=&
\cov\left(B_{H_1}(s^{H_2/H_1}),B_{H_2}(t^{H_2/H_1})\right),\nonumber
\end{eqnarray}
where inequality (\ref{i.conv}) is due to the convexity of function $f(x)=x^{H_2/H_1}$
(recall that $H_1<H_2$).

{\underline{ Ad (i)}}. From Slepian inequality, we get for each $u\in \R$
\begin{eqnarray*}
\p\left(\sup_{t\in[0,T]} \sqrt{2}B_{H_2}(t)-at^{2H_2}>u\right)
&\le&
\p\left(\sup_{t\in[0,T]} \sqrt{2}B_{H_1}(t^{H_2/H_1})-at^{2H_2}>u\right)\\
&=&
\p\left(\sup_{t\in[0,T^{H_2/H_1}]} \sqrt{2}B_{H_1}(t)-at^{2H_1}>u\right).
\end{eqnarray*}
Using that by (2.3) in \cite{Adl90}, for $H=H_1,H_2$ and any $T>0$
\[
\lim_{u\to\infty}\frac{\log\left(    \p\left(\sup_{t\in[0,T]} \sqrt{2}B_{H}(t)-at^{2H}>u\right)   \right)}{u^2}=-\frac{1}{2T^{2H}},
\]
and hence for each $a\in \R$
\[
\e \left( \exp\left( \alpha \left(\sup_{t\in [0,T]}  \sqrt{2}B_{H}(t)-at^{2H} \right)   \right) \right)
<\infty,
\]
completes the proof of case {\it (i)}.

{\underline{ Ad (ii)}}.
In view
of Theorem 1 in \cite{HuP99}, for any $H\in(0,1)$
\[
\lim_{u\to\infty}\frac{\log\left(    \p\left(\sup_{t\in[0,T]} \sqrt{2}B_{H}(t)-at^{2H}>u\right)   \right)}{u}=-a.
\]
Thus, for all $\alpha < a$
\[
\e \left( \exp\left( \alpha \left( \sup_{t\in [0,\infty)}  \sqrt{2}B_{H_1}(t)-at^{2H_1}   \right) \right) \right)
< \infty.
\]
The rest of the proof follows by the same argument as the proof of {\it (i)}.
The proof is completed.
\end{proof}

%\subsection{Bounds on moments of the supremum of a Gaussian process at a random time}

\begin{lemma}\label{lem:moments_sup_stationary}
Let $W(t)$ be a centered Gaussian process with stationary increments, bounded sample paths a.s.
and $W(0)=0$ a.s. Then, for any $n\in\N$ there exists $C(n)>0$ such that
\[\e \sup_{t\in[0,T]} |W(t)|^n \leq
C(n) \cdot \left(T^{n+2}(\var W(1))^{n/2} + T^2\e\Big(\sup_{t\in[0,1]} W(t)\Big)^n\right).
\]
\end{lemma}

\begin{proof}[Proof of Lemma~\ref{lem:moments_sup_stationary}]
Without loss of generality, we suppose that $T\in \N$.
We note that
\begin{eqnarray*}
\e \sup_{t\in[0,T]} |W(t)|^n \leq T \sum_{i=0}^{T-1} \e \sup_{t\in[i,i+1]} |W(t)|^n.
\end{eqnarray*}
Then,  for $i=0,...,T-1$
\begin{eqnarray*}
\e \sup_{t\in[i,i+1]} |W(t)|^n
&\le&
2 \e \left( \sup_{t\in[i,i+1]} (W(t)-W(i))+W(i)\right) ^n\nonumber\\
&\le&
2\max\{1,2^{n-1}\} \e |W(i)|^n+ 2\max\{1,2^{n-1}\}\e \Big(\sup_{t\in[0,1]} W(t)\Big)^n.
\end{eqnarray*}
Moreover, for $i\ge 1$
\begin{eqnarray*}
\e (W(i))^2
&=&
\e \left(     \sum_{k=0}^{i-1}(W(k+1)-W(k))  \right)^2\nonumber\\
&\le&
i \sum_{k=0}^{i-1} \e (W(1))^2 =  \e (W(1))^2  i^2
\end{eqnarray*}
and since $\e |W(i)|^n = C_n \cdot (\e|W(i)|^2)^{n/2}$, where $C_n := \e|\mathcal N(0,1))|^n = \frac{2^{n/2}\Gamma((n+1)/2)}{\sqrt{\pi}}$, then
\begin{align*}
\e |W(i)|^n \leq C_n(\var W(1))^{n/2}i^n.
\end{align*}
Finally, we have
\begin{align*}
\e \sup_{t\in[0,T]} |W(t)|^n \leq 2\max\{1,2^{n-1}\}T \left(\sum_{i=1}^{T-1}\left( C_n(\var W(1))^{n/2}i^n+\e \Big(\sup_{t\in[0,1]} W(t)\Big)^n\right)\right),
\end{align*}
which completes the proof.
\end{proof}

\begin{lemma}\label{lem:sup_moments}
Let $\{W(t); t\in[0,1]\}$ be a centered Gaussian process. If there exist $A>0$, $H\in(0,1]$ such that $\e |W(t)-W(s)|^2\leq A|t-s|^{2H}$ for all $t,s\in[0,T]$, then for every $n\in\Z_+$ there exists a constant $C(n)>0$ such that
\begin{align*}
\e \sup_{t\in[0,1]} |W(t)|^n \leq  C(n) \cdot A^n(1 + \mu^{n-1}(H) + \mu^n(H))
\end{align*}
where $\mu(H) := \e \sup_{t\in[0,1]} B_H(t)$.
\end{lemma}

\begin{proof}[Proof of Lemma~\ref{lem:sup_moments}]
When $n=1$ then the above holds by virtue of Sudakov-Fernique inequality with $C(1) = 2$, i.e. we have
\begin{align}\label{eq:sudakov-fernique_bound}
m := \e \sup_{t\in[0,1]} |W(t)| \leq \e \sup_{t\in[0,1]}\{W(t)\} + \e \sup_{t\in[0,1]}\{ -W(t)\} \leq 2C\mu(H).
\end{align}
Now, using Borell's inequality, with $\sigma^2 := \sup_{t\in[0,1]}\e W^2(t)$ we have
\begin{align*}
\e \sup_{t\in[0,1]} |W(t)|^n & = \int_0^\infty \p\left(\sup_{t\in[0,1]} |W(t)|^n > u\right){\rm d}u \\
& = \int_0^{m^n} \p\left(\sup_{t\in[0,1]} |W(t)| > u^{1/n}\right){\rm d}u + \int_{m^n}^\infty \p\left(\sup_{t\in[0,1]} |W(t)| > u^{1/n}\right){\rm d}u,
\end{align*}
where we used the substitution $x := (u^{1/p}-m)/\sigma$. The first term above is bounded by $m^n$. Using Borell inequality, we can bound the second term above with
\begin{align*}
\int_{m^p}^\infty \exp\left\{-\frac{(u^{1/n} - m)^2}{2\sigma^2}\right\}{\rm d}u & = n\sigma \int_0^\infty(\sigma x + m)^{n-1}e^{-x^2/2}{\rm d}u \\
& \leq  n\sigma \int_0^\infty(\max\{1,2^{n-2}\}((\sigma x)^{n-1} + m^{n-1})e^{-x^2/2}{\rm d}u \\
& \leq n(\sigma^n+\sigma m^{n-1})\int_0^\infty(x^{n-1} + 1)e^{-x^2/2}{\rm d}u.
\end{align*}
Since the integral above is finite and depends only on $n$, then the claim follows from \eqref{eq:sudakov-fernique_bound} and the fact that $\sigma \leq C$.
\end{proof}

\begin{lemma}\label{lem:moments_sup_random_time}
Let $W(t)$ be a centered Gaussian process with stationary increments and $\tau\geq0$ be a random time. If there exist $\gamma,\beta,C>0$ satisfying $\p(\tau > T) \leq Ce^{-\gamma T^\beta}$ for all $T>0$, then for each $n\in\N$ there exists $C'>0$ depending only on $n, \gamma,\beta,C$ such that
\[\e |W(\tau)|^n \leq C'\left((\var W(1))^{n/2} + \e\Big(\sup_{t\in[0,1]} W(t)\Big)^n\right).\]
\end{lemma}
\begin{proof}
See that
\begin{align*}
\e |W(\tau)|^n & = \sum_{n=0}^\infty \e\Big(|W(\tau)|^n; \tau\in[k,k+1])\Big) \\
& \leq \sum_{k=0}^\infty \e\Big(\sup_{t\in[0,k+1]} |W(t)|^n; \tau\geq k)\Big) \\
& \leq \sum_{k=0}^\infty \Big[\e\Big(\sup_{t\in[0,k+1]} |W(t)|^{2n}\Big)\p\big(\tau\geq k)\big)\Big]^{1/2},
\end{align*}
where in the last line we used H\"{o}lder's inequality. Using Lemma~\ref{lem:moments_sup_stationary} we further find that there exists a constant $C'>0$ such that
\begin{align*}
\e |W(\tau)|^n \leq C'\left((\var W(1))^{n/2} + \e\Big(\sup_{t\in[0,1]} W(t)\Big)^n\right)\cdot \sum_{k=0}^\infty (k+1)^{n+2}e^{-\gamma k^\beta},
\end{align*}
which completes the proof because the sum above is finite.
\end{proof}

%---------------------------------------------------------
%---------------------------------------------------------
%---------------------------------------------------------

\section{Calculation of integrals in Corollary~\ref{cor:explicit_expressions}}\label{appendix:calculations}
In the proof of Corollary~\ref{cor:explicit_expressions} one needs to handle complicated integrals. For convenience we present all these calculations here.
\begin{proof}[Proof of Corollary~\ref{cor:explicit_expressions}(i).]
Since
\begin{align*}
p(t,y;T,0) = \frac{y \exp\{-y^2/(2t)\}}{\pi t^{3/2}\sqrt{T-t}},
\end{align*}
cf.~\eqref{eq:density:T} (see also \cite[2.1.13.4]{borodin2002handbook}), then according to Theorem~\ref{thm:supremum_derivative_limit} we have $\mathscr M'_{1/2}(T,0) = \mathcal I_1 - \mathcal I_2$, where
\begin{align*}
\mathcal I_1 := \int_0^T\int_0^\infty y(1+\log(t))\frac{y \exp\{-\frac{y^2}{2t}\}}{\pi t^{3/2}\sqrt{T-t}}{\rm d}y{\rm d}t, \quad \mathcal I_2 :=  \int_0^T\int_0^\infty I(t,y) \frac{y \exp\{-\frac{y^2}{2t}\}}{\pi t^{3/2}\sqrt{T-t}}{\rm d}y{\rm d}t.
\end{align*}
Till this end we show that
\begin{equation}\label{eq:I1I2_to_show_1}
\mathcal I_1 = \sqrt{\frac{2T}{\pi}} \cdot (\log(T) + 2\log(2) - 1), \quad \mathcal I_2 = \sqrt{\frac{2T}{\pi}} \left(1 + 2\log(2)\right),
\end{equation}
which completes the proof.
We first calculate $\mathcal I_1$. After substitution $y = \sqrt{t}z$ we have
\begin{align*}
\mathcal I_1 = \sqrt{\frac{2}{\pi}}\int_0^T \frac{1+\log(t)}{\sqrt{T-t}} {\rm d}t\int_0^\infty \frac{z^2}{\sqrt{2\pi}}e^{-\frac{z^2}{2}}{\rm d}z =\sqrt{\frac{2}{\pi}} \cdot \left(2\sqrt{T} + \int_0^T \frac{\log(t)}{\sqrt{T-t}} {\rm d}t\right) \cdot \frac{1}{2}.
\end{align*}
Now, using substitution $t = Ts$ we obtain
\begin{eqnarray*}
  \lefteqn{\int_0^T \frac{\log(t)}{\sqrt{T-t}} {\rm d}t}\\
  &&= \sqrt{T}\int_0^1\frac{\log(sT)}{\sqrt{1-t}}{\rm d}s = 2\sqrt{T}\log(T) + \sqrt{T}\int_0^1\frac{\log(s)}{\sqrt{1-s}}{\rm d}s = \sqrt{T}(2\log(T) + 4\log(2)-4),
\end{eqnarray*}
where we used that $\int_0^1\frac{x\log(x)}{\sqrt{1-x^2}}{\rm d}x = \log(2)-1$; see formula 4.241-2 in \cite{integral_table}. This leads to the formula for $\mathcal I_1$ in \eqref{eq:I1I2_to_show_1}.

We now calculate $\mathcal I_2$. Using e.g. \cite[5-(A6)]{ng1969table} we find that for any $b\in\R$ we have
\begin{align*}
\int_0^\infty x^2{\rm e}^{-\frac{(x+b)^2}{2}}{\rm d}x = -be^{-b^2/2} + (1+b^2)\sqrt{\pi/2}\cdot \erfc(b/\sqrt{2})
\end{align*}
so, noting that for any $x\in\R$ we have $\erfc(-x)-\erfc(x) = 2\erf(x)$ and any $b\in\R$
\begin{equation}\label{eq:integral_b}
\int_0^\infty x^2\left(e^{-\frac{(x-b)^2}{2}}-e^{-\frac{(x+b)^2}{2}}\right){\rm d}x = 2be^{-b^2/2} + (1+b^2)\sqrt{2\pi}\cdot \erf(b/\sqrt{2}).
\end{equation}
Now, see that
\begin{align*}
\mathcal I_2 :=  \int\limits_0^T\int\limits_0^\infty \int\limits_0^\infty\int\limits_0^\infty \frac{\sqrt{2}tx^2}{yq(1+q^2)^2\sqrt{\pi}}\left(e^{-(x-yq/\sqrt{t})^2/2} - e^{-(x+yq/\sqrt{t})^2/2}\right) \frac{y \exp\{-\frac{y^2}{2t}\}}{\pi t^{3/2}\sqrt{T-t}}{\rm d}x{\rm d}q{\rm d}y{\rm d}t.
\end{align*}
The substitution $y = \sqrt{t}z$ yields
\begin{align*}
\mathcal I_2 :=  \frac{\sqrt{2}}{\pi^{3/2}}\int\limits_0^T \frac{1}{\sqrt{T-t}}{\rm d}t \cdot \int\limits_0^\infty \int\limits_0^\infty\int\limits_0^\infty \frac{x^2}{q(1+q^2)^2}\left(e^{-\frac{(x-zq)^2}{2}} - e^{-\frac{(x+zq)^2}{2}}\right) e^{-\frac{z^2}{2}}{\rm d}x{\rm d}q{\rm d}z.
\end{align*}
Using \eqref{eq:integral_b} we obtain
\begin{align*}
\int_0^\infty x^2\left(e^{-\frac{(x-zq)^2}{2}} - e^{-\frac{(x+zq)^2}{2}}\right) {\rm d}x = 2zqe^{-z^2q/2} + (1+z^2q^2)\sqrt{2\pi}\cdot \erf(zq/\sqrt{2})
\end{align*}
and thus
\begin{align*}
\mathcal I_2 :=  \sqrt{\frac{2T}{\pi}}\cdot \frac{2}{\pi}\int\limits_0^\infty \frac{1}{q(1+q^2)^2}(J_1(q) + J_2(q) + J_3(q)){\rm d}q,
\end{align*}
where using the tables of integrals \cite[4.3.4 and 4.3.6]{ng1969table} and the fact that $\pi - 2\arctan(1/x) = 2\arctan(x)$, we find that
\begin{align*}
J_1(q) & := 2q\int_0^\infty ze^{-z^2(1+q^2)/2}{\rm d}z = \frac{2q}{1+q^2},\\
J_2(q) & := \sqrt{2\pi}\int_0^\infty\erf(zq/\sqrt{2})e^{-z^2/2}{\rm d}z = 2\arctan(q),\\
J_3(q) & := q^2\sqrt{2\pi}\int_0^\infty z^2\erf(zq/\sqrt{2}) e^{-\frac{z^2}{2}}{\rm d}z = 2q^2\left(\arctan(q) + \frac{q}{1+q^2}\right).
\end{align*}
Finally, we have
\begin{align*}
\int_0^\infty \frac{J_1(q)}{q(1+q^2)^2}{\rm d}q & = 2\int_0^\infty \frac{1}{(1+q^2)^3}{\rm d}q \\
\int_0^\infty \frac{J_2(q)}{q(1+q^2)^2}{\rm d}q & = 2\int_0^\infty \frac{\arctan(q)}{q(1+q^2)^2}{\rm d}q\\
\int_0^\infty \frac{J_3(q)}{q(1+q^2)^2}{\rm d}q & = 2\int_0^\infty \frac{q\arctan(q)}{(1+q^2)^2}{\rm d}q + 2\int_0^\infty \frac{q^2}{(1+q^2)^3}{\rm d}q.
\end{align*}
We now use $\int_0^\infty \frac{q\arctan(q)}{(1+q^2)^2}{\rm d}q = \frac{\pi}{8}$, see e.g. formula 4.535-11 in \cite{integral_table}
and the identity
\begin{equation}\label{eq:rational_function_identity}
\frac{1}{q(1+q^2)^2} = \frac{1}{q(1+q^2)} - \frac{q}{(1+q^2)^2},
\end{equation}
to working out integral  $\int_0^\infty \frac{\arctan(q)}{q(1+q^2)^2}{\rm d}q$ with the use of formulas 4.531-7, and 4.535-11 in \cite{integral_table}. After the above calculations, we arrive at the formula for $\mathcal I_2$ in \eqref{eq:I1I2_to_show_1}. This completes the proof.
\end{proof}

\begin{proof}[Proof of Corollary~\ref{cor:explicit_expressions}(ii).]
Since, for $a>0$ we have
\begin{align*}
p(t,y,\infty,a) = \frac{\sqrt{2}\,ay\exp\left\{-\tfrac{(y+ta)^2}{2t}\right\}}{t^{3/2}\sqrt{\pi}},
\end{align*}
see \eqref{eq:density:infty}, then according to Theorem~\ref{thm:supremum_derivative_limit} we have $\mathscr M'_{1/2}(T,0) = \mathcal I_1 - \mathcal I_2$, where
\begin{align*}
\mathcal I_1 &:= \int_0^\infty\int_0^\infty (y(1+\log(t))+at\log(t))\frac{\sqrt{2}\,ay\exp\left\{-\tfrac{(y+ta)^2}{2t}\right\}}{t^{3/2}\sqrt{\pi}}{\rm d}y{\rm d}t, \\
\mathcal I_2 &:=  \int_0^\infty\int_0^\infty I(t,y) \frac{\sqrt{2}\,ay\exp\left\{-\tfrac{(y+ta)^2}{2t}\right\}}{t^{3/2}\sqrt{\pi}}{\rm d}y{\rm d}t.
\end{align*}
At the end, we show that
\begin{equation}\label{eq:I1I2_to_show_2}
\mathcal I_1 = \frac{3}{2a} - \frac{1}{a}\left(\gamma_{\rm E} - \log(2a^2)\right), \quad \mathcal I_2 = \frac{3}{2a},
\end{equation}
which yields the result. First we calculate $\mathcal I_1$. After substitution $y = tz$, we obtain
\begin{align*}
\mathcal I_1 &:= \sqrt{\frac{2}{\pi}} \cdot a \int_0^\infty\int_0^\infty (z+(z+a)\log(t))\cdot z t^{3/2}\exp\left\{-\tfrac{(z+a)^2}{2} \cdot t\right\}{\rm d}t{\rm d}z \\
& = \sqrt{\frac{2}{\pi}} \cdot a \int_0^\infty \left(z^2 J_1(z) + z(z+a)J_2(z)\right){\rm d}z,
\end{align*}
where
\begin{align*}
 J_1(z) := \int_0^\infty t^{3/2}\exp\left\{-t\cdot \tfrac{(z+a)^2}{2}\right\}{\rm d}t, \quad J_2(z) := \int_0^\infty \log(t)\cdot t^{3/2}\exp\left\{-t\cdot \tfrac{(z+a)^2}{2}\right\}{\rm d}t.
\end{align*}
The integrals $J_1$ and $J_2$ can be calculated explicitly using formulas 3.381-4 and 4.352-1 in \cite{integral_table} respectively and after some algebraic manipulations we arrive at
\begin{align*}
\mathcal I_1 = \sqrt{\frac{2}{\pi}} \cdot a \int_0^\infty \left(\frac{3z^2\sqrt{2\pi}}{(a+z)^5} -\frac{\sqrt{2\pi}z}{(a + z)^4} \cdot \left(-8 + 3\gamma_{\rm E} + \log(8) + 6\log(a+z)\right)\right){\rm d}z.
\end{align*}
The calculation of the integrals of rational functions above is standard. It is left to find the integral of the product of rational function and logarithm, i.e.
\begin{align*}
\int_0^\infty \frac{z}{(a+z)^4}\log(a+z){\rm d}z = a^{-2}\int_0^\infty \frac{y}{(1+y)^4}\left(\log(a) + \log(1+y)\right){\rm d}y
\end{align*}
where we used substitution $z = ay$. Again, the integral of the rational function can be found, in particular, we have $\int_0^\infty \frac{y}{(1+y)^4}{\rm d}y = \tfrac{1}{6}$, so it is left to find
\begin{align*}
\int_0^\infty \frac{y\log(1+y)}{(1+y)^4}{\rm d}y = \int_0^1 w(1-w)\log(w){\rm d}w = -\frac{5}{36},
\end{align*}
where we substituted $w = (1+y)^{-1}$ and used formula 4.253-1 in \cite{integral_table}. Finally, simple algebraic manipulations yield the value of $\mathcal I_1$ in \eqref{eq:I1I2_to_show_2}.

We now calculate the value of $\mathcal I_2$. After substitution $y = tz$ and $x=\sqrt{t}w$ we obtain:
\begin{align*}
\mathcal I_2 = \frac{2a}{\pi}\int_0^\infty\int_0^\infty\int_0^\infty \frac{w^2(J_{1}(w,z,q)-J_{2}(w,z,q))}{q(1+q^2)^2} {\rm d}w{\rm d}z{\rm d}q,
\end{align*}
where
\begin{align*}
J_{1}(w,z,q) = \int_0^\infty t^2 \cdot e^{-t\cdot\frac{(a+z)^2 + (w-qz)^2}{2}}{\rm d}t, \quad J_{2}(w,z,q) = \int_0^\infty t^2 \cdot e^{-t\cdot\frac{(a+z)^2 + (w+qz)^2}{2}}{\rm d}t.
\end{align*}
Since $\int_0^\infty t^2e^{-\mu t} = \mu^{-3}\Gamma(3)$, then we obtain
\begin{align*}
\mathcal I_2 =\frac{32a}{\pi} \int_0^\infty\int_0^\infty\int_0^\infty \frac{J_1'(z,q) - J_2'(z,q)}{q(1+q^2)^2}{\rm d}z{\rm d}q,
\end{align*}
where
\begin{align*}
J'_{1}(z,q) = \int_0^\infty \frac{w^2}{((a+z)^2+(w-qz)^2)^3}{\rm d}w,  \quad J'_{2}(z,q) = \int_0^\infty \frac{w^2}{((a+z)^2+(w+qz)^2)^3}{\rm d}w.
\end{align*}
Functions $J_1'(z,q)$ and $J_2'(z,q)$ are definite integrals of rational functions, so it is well-known how to find them. Below, we write the explicit formula for $\mathcal I_2$
\begin{align*}
\mathcal I_2 = \frac{8a}{\pi}\int_0^\infty\int_0^\infty\frac{3qz(a+z) + ((a+z)^2 + 3q^2z^2)\arctan(\tfrac{qz}{a+z})}{q(1+q^2)^2 (a+z)^5}{\rm d}q{\rm d}z.
\end{align*}
After substitution $z = ay$ we obtain
\begin{align*}
\mathcal I_2 = \frac{8}{\pi a}\int_0^\infty\int_0^\infty\frac{3qy(1+y) + ((1+y)^2 + 3q^2y^2)\arctan(\tfrac{qy}{1+y})}{q(1+q^2)^2(1+y)^5}{\rm d}q{\rm d}y,
\end{align*}
and further substituting $z = \frac{y}{1+y}$ we find that
\begin{align*}
\mathcal I_2 = \frac{8}{\pi a}\int_0^1(1-z)\cdot \left(I_1(z) + I_2(z) + I_3(z)\right){\rm d}z,
\end{align*}
where
\begin{align*}
I_1(z) & := 3z\int_0^\infty \frac{1}{(1+q^2)^2}{\rm d}q = \frac{\pi}{4} \cdot 3z,\\
I_2(z) & := \int_0^\infty \frac{\arctan(qz)}{q(1+q^2)^2}{\rm d}q = \frac{\pi}{4} \cdot \left(2\log(1+z) - \frac{z}{1+z}\right),\\
I_3(z) & := 3z^2\int_0^\infty \frac{q\arctan(qz)}{(1+q^2)^2}{\rm d}q = \frac{\pi}{4} \cdot \frac{3z^3}{1+z},
\end{align*}
where the formula for $I_2(z)$ was found using \eqref{eq:rational_function_identity} and formulas 4.535-7 and 4.535-11 in \cite{integral_table}, while $I_3(z)$ follows from 4.535-11 in \cite{integral_table}. This gives us
\begin{align*}
\mathcal I_2 & = \frac{2}{a} \int_0^1 \left(3z(1-z) + 2(1-z)\log(1+z) + \frac{z(1-z)(3z^2-1)}{1+z}\right){\rm d}z\\
& = \frac{2}{a} \cdot \left(\frac{1}{2} +\log(16)-\frac{5}{2} + \frac{11}{4}-\log(16)\right) = \frac{3}{2a},
\end{align*}
which completes the proof.
\end{proof}

\begin{proof}[Proof of Corollary~\ref{cor:explicit_expressions}(iii).]
The following calculations are very similar to the ones did in the proof of Corollary~\ref{cor:explicit_expressions}(ii). For $a>1$ we have
\begin{align*}
p(t,y,\infty,a/\sqrt{2}) = \frac{ay\exp\left\{-\tfrac{(y+ta/\sqrt{2})^2}{2t}\right\}}{t^{3/2}\sqrt{\pi}},
\end{align*}
see \eqref{eq:density:infty}. Therefore, according to Theorem~\ref{thm:supremum_derivative_limit} we have $\mathscr P'_{1/2}(T,0) = \mathcal I_1 - \mathcal I_2$, where
\begin{align*}
\mathcal I_1 &:=\sqrt{\frac{2}{\pi}} \cdot a\int_0^\infty\int_0^\infty (y(1+\log(t))-\tfrac{a}{\sqrt{2}}t\log(t))e^{\sqrt{2}y}\cdot\frac{y\exp\left\{-\tfrac{(y+ta/\sqrt{2})^2}{2t}\right\}}{t^{3/2}\sqrt{\pi}}{\rm d}y{\rm d}t, \\
\mathcal I_2 &:=  \sqrt{\frac{2}{\pi}} \cdot a\int_0^\infty\int_0^\infty I(t,y)e^{\sqrt{2}y}\cdot \frac{y\exp\left\{-\tfrac{(y+ta/\sqrt{2})^2}{2t}\right\}}{t^{3/2}}{\rm d}y{\rm d}t.
\end{align*}
In the following $\acoth(\cdot)$ is the inverse hyperbolic cotangent function, i.e. for $|z|>1$
\begin{align*}
\acoth(z) = \frac{1}{2}\Big(\log\big(1+\tfrac{1}{z}\big) - \log\big(1-\tfrac{1}{z}\big)\Big) = \frac{1}{2}\log\big(\tfrac{z+1}{z-1}\big).
\end{align*}
Till this end, we show that
\begin{equation}\label{eq:I1I2_to_show_3}
\mathcal I_1 = \frac{a\left(3 - 2a - 4(a-1)^2\acoth(1-2a)\right)}{(a-1)^2}, \quad \mathcal I_2 = \frac{a\left(1 - 4(a-1)\acoth(1-2a)\right)}{(a-1)^2}
\end{equation}
which yields the result. First we calculate $\mathcal I_1$. After substitution $y = tz$, we obtain
\begin{align*}
\mathcal I_1 &:=\sqrt{\frac{2}{\pi}} \cdot a \int_0^\infty\int_0^\infty (z+(z-\tfrac{a}{\sqrt{2}})\log(t))\cdot z t^{3/2}\exp\left\{-\tfrac{(z+a/\sqrt{2})^2-2\sqrt{2}z}{2} \cdot t\right\}{\rm d}t{\rm d}z.
\end{align*}
Define $z_1^2 := (z+\frac{a}{\sqrt{2}})^2 - 2\sqrt{2}z = (z+\frac{a-2}{\sqrt{2}})^2 + 2(a-1)$. We have
\begin{align*}
\mathcal I_1 =\sqrt{\frac{2}{\pi}} \cdot a \int_0^\infty \left(z^2 J_1(z) + z(z-\tfrac{a}{\sqrt{2}})J_2(z)\right){\rm d}z,
\end{align*}
where
\begin{align*}
 J_1(z) := \int_0^\infty t^{3/2}e^{-t\cdot z_1^2/2}{\rm d}t, \quad J_2(z) := \int_0^\infty \log(t)\cdot t^{3/2}e^{-t\cdot z_1^2/2}{\rm d}t.
\end{align*}
The integrals $J_1$ and $J_2$ can be calculated explicitly using the formulas 3.381-4 and 4.352-1 in \cite{integral_table}, which gives us
\begin{align*}
\mathcal I_1 = 2a\int_0^\infty \left(\frac{z^2(11-(3\gamma_{\rm E} + 3\log(2) + 6\log(z_1)))}{z_1^5} - \frac{az(8-(3\gamma_{\rm E} + 3\log(2) + 6\log(z_1)))}{\sqrt{2}z_1^5}\right){\rm d}z.
\end{align*}
Finally, we calculate the quantity above using the definite integrals below
\begin{align*}
\int_0^\infty \frac{z}{z_1^5}{\rm d}z & = \frac{1}{3\sqrt{2}(a-1)^2a} \\
\int_0^\infty \frac{z^2}{z_1^5}{\rm d}z & = \frac{1}{6(a-1)^2}\\
\int_0^\infty \frac{z\log(z_1)}{z_1^5}{\rm d}z & = \frac{8-3\log(2) - 3a+3(a-2)a\log\big(\tfrac{a}{a-1}\big)+6\log(a)}{18\sqrt{2}(a-1)^2a}\\
\int_0^\infty \frac{z^2\log(z_1)}{z_1^5}{\rm d}z & = \frac{4+6a-6\log(2)+6(2+(a-2)a)\log(a-1) - 6(a-2)a\log(a)}{72(a-1)^2},
\end{align*}
which, after algebraic manipulations yields the formula for $\mathcal I_1$ in \eqref{eq:I1I2_to_show_3}.

We now proceed to calculate the value of $\mathcal I_2$. After substitution $y = tz$ and $x=\sqrt{t}w$ we obtain:
\begin{align*}
\mathcal I_2 = \frac{2a}{\pi}\int_0^\infty\int_0^\infty\int_0^\infty \frac{w^2(J_{1}(w,z,q)-J_{2}(w,z,q))}{q(1+q^2)^2} {\rm d}w{\rm d}z{\rm d}q,
\end{align*}
where
\begin{align*}
J_{1}(w,z,q) = \int_0^\infty t^2 \cdot e^{-t\cdot\frac{\left(z+\tfrac{a}{\sqrt{2}}\right)^2 -2\sqrt{2}z + (w-qz)^2}{2}}{\rm d}t, \quad J_{2}(w,z,q) = \int_0^\infty t^2 \cdot e^{-t\cdot\frac{\left(z+\tfrac{a}{\sqrt{2}}\right)^2-2\sqrt{2}z + (w+qz)^2}{2}}{\rm d}t.
\end{align*}
Since $\int_0^\infty t^2e^{-\mu t} = \mu^{-3}\Gamma(3)$, then
\begin{align*}
\mathcal I_2 =\frac{32a}{\pi} \int_0^\infty\int_0^\infty\int_0^\infty \frac{J_1'(z,q) - J_2'(z,q)}{q(1+q^2)^2}{\rm d}z{\rm d}q,
\end{align*}
where
\begin{align*}
J'_{1}(z,q) = \int_0^\infty \frac{w^2}{(z_1^2+(w-qz)^2)^3}{\rm d}w,  \quad J'_{2}(z,q) = \int_0^\infty \frac{w^2}{(z_1^2+(w+qz)^2)^3}{\rm d}w,
\end{align*}
with $z_1^2 := (z+\frac{a}{\sqrt{2}})^2 - 2\sqrt{2}z = (z+\frac{a-2}{\sqrt{2}})^2 + 2(a-1)$. Functions $J_1'(z,q)$ and $J_2'(z,q)$ are definite integrals of rational functions, so it is well-known how to find them. Now, we write
\begin{align*}
\mathcal I_2 = \frac{8a}{\pi}\int_0^\infty\int_0^\infty\frac{3qz z_1 + (z_1^2 + 3q^2z^2)\arctan(\tfrac{qz}{z_1})}{q(1+q^2)^2 z_1^5}{\rm d}q{\rm d}z.
\end{align*}
We thus have
\begin{align*}
\mathcal I_2 = \frac{8a}{\pi}\int_0^\infty \frac{I_1(z)+I_2(z)+I_3(z)}{z_1^5}{\rm d}z,
\end{align*}
where
\begin{align*}
I_1(z) & := 3zz_1\int_0^\infty \frac{1}{(1+q^2)^2}{\rm d}q = \frac{\pi}{4} \cdot 3zz_1,\\
I_2(z) & := z_1^2\int_0^\infty \frac{\arctan(\tfrac{qz}{z_1})}{q(1+q^2)^2}{\rm d}q = \frac{\pi}{4} \cdot z_1^2\left(2\log(1+\tfrac{z}{z_1}) - \frac{z}{z+z_1}\right)\\
I_3(z) & := 3z^2\int_0^\infty \frac{q\arctan(\tfrac{qz}{z_1})}{(1+q^2)^2}{\rm d}q = \frac{\pi}{4} \cdot \frac{3z^3}{z+z_1}
\end{align*}
where the formula from $I_2(z)$ follows from \eqref{eq:rational_function_identity} and formulas 4.535-7 and 4.535-11 in \cite{integral_table}, while $I_3(z)$ was found using formula 4.535-11. This gives us
\begin{align*}
\mathcal I_2 & = 2a \cdot \left(3\int_0^\infty\frac{z}{z_1^4}{\rm d}z + 2\int_0^\infty \frac{\log(1+\tfrac{z}{z_1})}{z_1^3}{\rm d}z - \int_0^\infty \frac{z}{z_1^3(z+z_1)}{\rm d}z + 3\int_0^\infty  \frac{z^3}{z_1^5(z+z_1)}{\rm d}z \right).
\end{align*}
Finally, we calculate the quantity above using the definite integrals below
\begin{align*}
\int_0^\infty\frac{z}{z_1^4}{\rm d}z & = \frac{2\sqrt{a-1}+(a-2)\left(\atan\big(\tfrac{a-2}{2\sqrt{a-1}}\big)-\tfrac{\pi}{2}\right)}{8(a-1)^{3/2}}\\
\int_0^\infty \frac{\log(1+\tfrac{z}{z_1})}{z_1^3}{\rm d}z & = \frac{-a+4(a-1)\log(2) + (a-2)\sqrt{a-1}\left(\atan\big(\tfrac{a-2}{2\sqrt{a-1}}\big)-\tfrac{\pi}{2}\right)-(a-2)^2\acoth(1-2a)}{(a-1)a^2}\\
\int_0^\infty \frac{z}{z_1^3(z+z_1)}{\rm d}z & = \frac{a - (a-2)\sqrt{a-1}\left(\atan\big(\tfrac{a-2}{2\sqrt{a-1}}\big)-\tfrac{\pi}{2}\right) - 4(a-1)\Big(\acoth(1-2a)+\log(2)\Big)}{(a-1)a^2}\\
\int_0^\infty  \frac{z^3}{z_1^5(z+z_1)}{\rm d}z & = \frac{2a(-3a^2+17a-12) - 3(a-2)\sqrt{a-1}(a^2+8a-8)\left(\atan\big(\tfrac{a-2}{2\sqrt{a-1}}\big) - \tfrac{\pi}{2} \right)}{24(a-1)^{2}a^2} \\
& \quad\quad -\frac{4\Big(\acoth(1-2a) + \log(2)\Big)}{a^2},
\end{align*}
which, after some algebraic manipulations yields \eqref{eq:I1I2_to_show_3} for $\mathcal I_2$. This completes the proof.
\end{proof}

%---------------------------------------------------------
%---------------------------------------------------------
%---------------------------------------------------------

\section*{Acknowledgements}

KB's research was funded by SNSF Grant 200021-196888. KD and TR were partially supported by NCN Grant No 2018/31/B/ST1/00370
(2019-2022).

\medskip
\bibliographystyle{plain}

\begin{thebibliography}{10}

\bibitem{Adl90}
Robert~J. Adler.
\newblock {\em An introduction to continuity, extrema, and related topics for
  general {G}aussian processes}, volume~12 of {\em Institute of Mathematical
  Statistics Lecture Notes---Monograph Series}.
\newblock Institute of Mathematical Statistics, Hayward, CA, 1990.

\bibitem{AdT07}
Robert~J. Adler and Jonathan~E. Taylor.
\newblock {\em Random fields and geometry}.
\newblock Springer Monographs in Mathematics. Springer, New York, 2007.

\bibitem{asmussen1995discretization}
S{\o}ren Asmussen, Peter Glynn, and Jim Pitman.
\newblock Discretization error in simulation of one-dimensional reflecting
  {B}rownian motion.
\newblock {\em Ann. Appl. Probab.}, 5(4):875--896, 1995.

\bibitem{BDHL18}
Long Bai, Krzysztof D\c{e}bicki, Enkelejd Hashorva, and Li~Luo.
\newblock On generalised {P}iterbarg constants.
\newblock {\em Methodol. Comput. Appl. Probab.}, 20(1):137--164, 2018.

\bibitem{BDM21}
Krzysztof Bisewski, Krzysztof D\c{e}bicki, and Michel Mandjes.
\newblock Bounds for expected supremum of fractional {B}rownian motion with
  drift.
\newblock {\em J. Appl. Probab.}, 58(2):411--427, 2021.

\bibitem{borodin2002handbook}
Andrei~N. Borodin and Paavo Salminen.
\newblock {\em Handbook of {B}rownian motion---facts and formulae}.
\newblock Probability and its Applications. Birkh\"{a}user Verlag, Basel,
  second edition, 2002.

\bibitem{Bor17}
Konstantin Borovkov, Yuliya Mishura, Alexander Novikov, and Mikhail Zhitlukhin.
\newblock Bounds for expected maxima of {G}aussian processes and their discrete
  approximations.
\newblock {\em Stochastics}, 89(1):21--37, 2017.

\bibitem{Bor18}
Konstantin Borovkov, Yuliya Mishura, Alexander Novikov, and Mikhail Zhitlukhin.
\newblock New and refined bounds for expected maxima of fractional {B}rownian
  motion.
\newblock {\em Statistics \& Probability Letters}, 137:142--147, 2018.

\bibitem{DeH20}
Krzysztof D\c{e}bicki and Enkelejd Hashorva.
\newblock Approximation of supremum of max-stable stationary processes \&
  {P}ickands constants.
\newblock {\em J. Theoret. Probab.}, 33(1):444--464, 2020.

\bibitem{DeK08}
Krzysztof D\c{e}bicki and Pawe{\l} Kisowski.
\newblock A note on upper estimates for {P}ickands constants.
\newblock {\em Statist. Probab. Lett.}, 78(14):2046--2051, 2008.

\bibitem{DeM15}
Krzysztof D\c{e}bicki and Michel Mandjes.
\newblock {\em Queues and {L}\'{e}vy fluctuation theory}.
\newblock Universitext. Springer, Cham, 2015.

\bibitem{DMR03}
Krzysztof D\c{e}bicki, Zbigniew Michna, and Tomasz Rolski.
\newblock Simulation of the asymptotic constant in some fluid models.
\newblock {\em Stoch. Models}, 19(3):407--423, 2003.

\bibitem{DelormeEtAl2017Pickands}
Mathieu Delorme, Alberto Rosso, and Kay~J\"{o}rg Wiese.
\newblock Pickands' constant at first order in an expansion around {B}rownian
  motion.
\newblock {\em J. Phys. A}, 50(16):16LT04, 13, 2017.

\bibitem{DelormeWiese2015Maximum}
Mathieu Delorme and Kay~J\"{o}rg Wiese.
\newblock Maximum of a fractional {B}rownian motion: analytic results from
  perturbation theory.
\newblock {\em Phys. Rev. Lett.}, 115(21):210601, 5, 2015.

\bibitem{DelormeandWiese2016Extreme}
Mathieu Delorme and Kay~J\"{o}rg Wiese.
\newblock Extreme-value statistics of fractional {B}rownian motion bridges.
\newblock {\em Phys. Rev. E}, 94(5):052105, 15, 2016.

\bibitem{delorme2016perturbative}
Mathieu Delorme and Kay~J{\"o}rg Wiese.
\newblock Perturbative expansion for the maximum of fractional brownian motion.
\newblock {\em Physical Review E}, 94(1):012134, 2016.

\bibitem{DiM15}
Antonius~B. Dieker and Thomas Mikosch.
\newblock Exact simulation of {B}rown-{R}esnick random fields at a finite
  number of locations.
\newblock {\em Extremes}, 18(2):301--314, 2015.

\bibitem{DiY14}
Antonius~B. Dieker and Benjamin Yakir.
\newblock On asymptotic constants in the theory of extremes for {G}aussian
  processes.
\newblock {\em Bernoulli}, 20(3):1600--1619, 2014.

\bibitem{Ferger1999uniqueness}
Dietmar Ferger.
\newblock On the uniqueness of maximizers of {M}arkov-{G}aussian processes.
\newblock {\em Statist. Probab. Lett.}, 45(1):71--77, 1999.

\bibitem{integral_table}
Izrail~S. Gradshteyn and Iosif~M. Ryzhik.
\newblock Table of integrals, series, and products.
\newblock Elsevier/Academic Press, Amsterdam, eighth edition, 2015.
\newblock Translated from the Russian, Translation edited and with a preface by
  Daniel Zwillinger and Victor Moll, Revised from the seventh edition
  [MR2360010].

\bibitem{harper2017}
Adam~J. Harper.
\newblock Pickands' constant {$H_\alpha$} does not equal
  {$1/\Gamma(1/\alpha)$}, for small {$\alpha$}.
\newblock {\em Bernoulli}, 23(1):582--602, 2017.

\bibitem{HuP99}
Jurg H\"{u}sler and Vladimir~I. Piterbarg.
\newblock Extremes of a certain class of {G}aussian processes.
\newblock {\em Stochastic Process. Appl.}, 83(2):257--271, 1999.

\bibitem{imhof1984density}
J.~P. Imhof.
\newblock Density factorizations for {B}rownian motion, meander and the
  three-dimensional {B}essel process, and applications.
\newblock {\em J. Appl. Probab.}, 21(3):500--510, 1984.

\bibitem{kallenberg2002}
Olav Kallenberg.
\newblock {\em Foundations of modern probability}.
\newblock Probability and its Applications (New York). Springer-Verlag, New
  York, second edition, 2002.

\bibitem{Lifshits1982absolute}
Mikhail~A. Lifshits.
\newblock Absolute continuity of the distributions of functionals of random
  processes.
\newblock {\em Teor. Veroyatnost. i Primenen.}, 27(3):559--566, 1982.

\bibitem{mandelbrot1968fractional}
Benoit~B. Mandelbrot and John~W. Van~Ness.
\newblock Fractional {B}rownian motions, fractional noises and applications.
\newblock {\em SIAM Rev.}, 10:422--437, 1968.

\bibitem{Man07}
Michel Mandjes.
\newblock {\em Large deviations for {G}aussian queues}.
\newblock John Wiley \& Sons, Ltd., Chichester, 2007.
\newblock Modelling communication networks.

\bibitem{mishura2008stochastic}
Yuliya~S. Mishura.
\newblock {\em Stochastic calculus for fractional {B}rownian motion and related
  processes}, volume 1929 of {\em Lecture Notes in Mathematics}.
\newblock Springer-Verlag, Berlin, 2008.

\bibitem{ng1969table}
Edward~W. Ng and Murray Geller.
\newblock A table of integrals of the error functions.
\newblock {\em J. Res. Nat. Bur. Standards Sect. B}, 73B:1--20, 1969.

\bibitem{peltier1995multifractional}
Romain-Fran{\c{c}}ois Peltier and Jacques~L{\'e}vy V{\'e}hel.
\newblock {\em Multifractional Brownian motion: definition and preliminary
  results}.
\newblock PhD thesis, INRIA, 1995.

\bibitem{Pit96}
Vladimir~I. Piterbarg.
\newblock {\em Asymptotic methods in the theory of {G}aussian processes and
  fields}, volume 148 of {\em Translations of Mathematical Monographs}.
\newblock American Mathematical Society, Providence, RI, 1996.
\newblock Translated from the Russian by V. V. Piterbarg, Revised by the
  author.

\bibitem{PiP78}
Vladimir~I. Piterbarg and Vladimir Prisiazhniuk.
\newblock Asymptotic analysis of the probability of large excursions for a
  nonstationary {G}aussian process.
\newblock {\em Teoriia Veroiatnostei i Matematicheskaia Statistika},
  (18):121--134, 1978.

\bibitem{samorodnitsky1994stable}
Gennady Samorodnitsky and Murad~S. Taqqu.
\newblock {\em Stable non-{G}aussian random processes}.
\newblock Stochastic Modeling. Chapman \& Hall, New York, 1994.
\newblock Stochastic models with infinite variance.

\bibitem{Seijo2011A}
Emilio Seijo and Bodhisattva Sen.
\newblock A continuous mapping theorem for the smallest argmax functional.
\newblock {\em Electron. J. Stat.}, 5:421--439, 2011.

\bibitem{Sha96}
Qi-Man Shao.
\newblock Bounds and estimators of a basic constant in extreme value theory of
  {G}aussian processes.
\newblock {\em Statist. Sinica}, 6(1):245--257, 1996.

\bibitem{shepp1979joint}
Lawrence~A. Shepp.
\newblock The joint density of the maximum and its location for a {W}iener
  process with drift.
\newblock {\em J. Appl. Probab.}, 16(2):423--427, 1979.

\bibitem{StoevTaqqu04}
Stilian~A. Stoev and Murad~S. Taqqu.
\newblock Stochastic properties of the linear multifractional stable motion.
\newblock {\em Adv. in Appl. Probab.}, 36(4):1085--1115, 2004.

\bibitem{StoevTaqqu05}
Stilian~A. Stoev and Murad~S. Taqqu.
\newblock Path properties of the linear multifractional stable motion.
\newblock {\em Fractals}, 13(2):157--178, 2005.

\bibitem{StoevTaqqu06}
Stilian~A. Stoev and Murad~S. Taqqu.
\newblock How rich is the class of multifractional {B}rownian motions?
\newblock {\em Stochastic Process. Appl.}, 116(2):200--221, 2006.

\bibitem{Vit96}
Richard~A. Vitale.
\newblock The {W}ills functional and {G}aussian processes.
\newblock {\em Ann. Probab.}, 24(4):2172--2178, 1996.

\bibitem{Mathematica}
{Wolfram Research Inc.}
\newblock Mathematica, {V}ersion 12.1, 2021.

\end{thebibliography}

\end{document}